\documentclass[a4paper,10pt]{article}

\usepackage{latexsym}
\usepackage{mathrsfs}
\usepackage{amsfonts,amsmath,amssymb}
\usepackage{indentfirst}
\usepackage{verbatim}
\usepackage[pdftex]{graphicx}
\usepackage{epstopdf}
\usepackage{stmaryrd}
\usepackage{multirow}
\usepackage{diagbox}
\usepackage{subcaption}
\usepackage{float}
\usepackage{booktabs}
\usepackage{tabularx} % For adjustable table width
\usepackage{algorithm}
\usepackage{algpseudocode}
\usepackage{bm}

\usepackage[utf8]{inputenc}
\usepackage[T1]{fontenc}

\newcommand{\bx}{\mbox{\boldmath{$x$}}}

\newcommand{\bb}{\mbox{\boldmath{$b$}}}

\newcommand{\bn}{\mbox{\boldmath{$n$}}}
\newcommand{\bu}{\mbox{\boldmath{$u$}}}

\newcommand{\bv}{\mbox{\boldmath{$v$}}}

\newcommand{\bw}{\mbox{\boldmath{$w$}}}

\newcommand{\bC}{\mbox{\boldmath{$C$}}}
\newcommand{\bX}{\mbox{\boldmath{$X$}}}
\newcommand{\bI}{\mbox{\boldmath{$I$}}}
\newcommand{\bB}{\mbox{\boldmath{$B$}}}
\newcommand{\bpsi}{\mbox{\boldmath{$\psi$}}}
\newcommand{\bc}{\mbox{\boldmath{$c$}}}
\newcommand{\bP}{\mbox{\boldmath{$P$}}}

\newcommand{\on}{{\rm on}}

\newcommand{\bxi}{\mbox{\boldmath{$\xi$}}}

\newcommand{\R}{\mathbb{R}}

\newcommand{\dd}{\mathrm{d}}

\newcommand{\diam}{\diamond}

\newcommand{\Gop}{\mathcal{G}}
\newcommand{\Bop}{\mathcal{B}}

\newtheorem{theorem}{Theorem}[section]
\newtheorem{lemma}[theorem]{Lemma}
\newtheorem{assumption}[theorem]{Assumption}
\newtheorem{corollary}[theorem]{Corollary}

\newtheorem{example}[theorem]{Example}

\newtheorem{remark}[theorem]{Remark}
\numberwithin{equation}{section}

\newenvironment{proof}[1][Proof]{\textbf{#1.} }
{\ \rule{0.75em}{0.75em}\smallskip}

\usepackage[pdftex,unicode,colorlinks,linkcolor=blue]{hyperref}

\begin{document}

\begin{center}
\Large\bf Randomized Neural Networks for Partial Differential Equation\\ on Static and Evolving Surfaces
\end{center}

\begin{center}
Jingbo Sun\footnote{School of Mathematics and Statistics, Xi'an Jiaotong University, Xi'an, Shaanxi 710049, P.R. China. E-mail: {\tt jingbosun@xjtu.edu.cn}.},
\quad 
Fei Wang\footnote{School of Mathematics and Statistics \& State Key Laboratory of Multiphase Flow in Power Engineering, Xi’an Jiaotong University, Xi’an, Shaanxi 710049, China. The work of this author was partially supported by the National Natural Science Foundation of China (Grant No. 92470115). Email: {\tt feiwang.xjtu@xjtu.edu.cn}.}
\end{center}

\medskip

\begin{quote}
\textbf{Abstract.} 
Surface partial differential equations arise in numerous scientific and engineering applications.
Their numerical solution on static and evolving surfaces remains challenging due to geometric complexity and, for evolving geometries,
the need for repeated mesh updates and geometry or solution transfer.
While neural-network-based methods offer mesh-free discretizations, approaches based on nonconvex training can be costly and may fail to
deliver high accuracy in practice.
In this work, we develop a randomized neural network (RaNN) method for solving PDEs on both static and evolving surfaces: the hidden-layer parameters are randomly generated and kept fixed,
and the output-layer coefficients are determined efficiently by solving a least-squares problem.
For static surfaces, we present formulations for parametrized surfaces, implicit level-set surfaces, and point-cloud
geometries, and provide a corresponding theoretical analysis for the parametrization-based formulation with interface compatibility.
For evolving surfaces with topology preserved over time, we introduce a RaNN-based strategy that learns the surface evolution through a
flow-map representation and then solves the surface PDE on a space--time collocation set, avoiding remeshing.
Extensive numerical experiments demonstrate broad applicability and favorable accuracy--efficiency performance on representative
benchmarks.
\end{quote}

\medskip

\section{Introduction}

Partial differential equations (PDEs) are fundamental tools for modeling a wide range of phenomena in physics, biology, and engineering. Many important PDEs are defined on curved manifolds and appear in diverse application areas such as image processing (\cite{Niang2012image,Biddle2013image,Lozes2014imge}), biomechanics (\cite{Stoop2015biomechanics,Mietke2019biomechanics}), and phase-field models (\cite{Gugenberger2008phase}). Explicitly incorporating surface geometry leads to more realistic and accurate simulations of these complex systems.

Despite their importance, solving PDEs on static and evolving surfaces remains highly challenging due to the geometric complexity of the underlying domains. Traditional methods, such as surface finite element methods (FEMs, \cite{Dziuk2007fem,Dziuk2007femp,Dziuk2008fempi,Elliott2012alefem,Dziuk2013fem}), address surface PDEs by discretizing the manifold using triangulations or other mesh structures. In contrast, mesh-free approaches like radial basis function methods (\cite{Fuselier2013rfm,Alvarez2018rbf,Shankar2020rbf,Wendland2020rbf}) offer high accuracy and geometric flexibility without requiring mesh generation. Another family of approaches, known as embedding methods, recasts PDEs posed on surfaces as equivalent volumetric PDEs in $\mathbb{R}^3$, such as the level set method (\cite{Bertalmio2001im}) and the closest point method (\cite{Ruuth2008close,Petras2016evolving,Petras2019close}).

Recent advances in machine learning have spurred the development of neural network-based methods for solving PDEs. Examples include Physics-Informed Neural Networks (PINNs) (\cite{Raissi2019PINN}), which solve PDEs in their strong form, and the Deep Ritz Method (\cite{Ee2018deepRitz}), which adopts a variational approach. Numerous other neural network-based frameworks have also emerged (\cite{Sirignano2018DGM,Liao2019deepRitzboundary,Jagtap2020pinn,Jagtap2020adappinn,Jagtap2020xpinn,Mao2020hsf,Zang2020adversarialnns,Wang2024tnn}). These methods exploit the universal approximation capabilities of neural networks (\cite{Cybenko1989appro,Hornik1989apro,Barron1993apro,Chen1995appro,Mhaskar1995appro,Lu2017appro,Ma2022appro}), including convergence rates independent of the dimensionality (\cite{Barron1993apro}). Several recent works have extended PINNs to surface PDEs (\cite{Fang2019sur_pinn,Fang2021sur_pinn,Tang2022sur_pinn}). However, PINNs typically reformulate even linear PDEs as nonconvex optimization problems, may suffer from optimization difficulties, which can limit accuracy and efficiency. A recent approach by Hu et al. (\cite{Hu2024pinn}) enhances both accuracy and efficiency by combining PINN and embedding techniques with shallow networks and Levenberg–Marquardt optimization (\cite{Marquardt1963lm}).

In this work, we pursue a different strategy by employing Randomized Neural Networks (RaNN) to solve PDEs on both static and evolving surfaces. RaNN has shown strong potential in volumetric PDEs and offers advantages in training simplicity and computational speed (\cite{Dwivedi2020elm,Dong2021locELM,Chen2022rfm,Shang2022DeepPetrov,Sun2024lrnndg,Li2025int,AGRANN2025}). The proposed method is mesh-free and applicable to a broad class of surfaces, including parametrizable surfaces, level-set-defined surfaces, and surfaces known only through scattered point data. RaNN was first introduced in~\cite{Pao1992RNN3,Pao1994RNN4,Igelnik1995RNN1,Igelnik1999RNN2}. While the architecture resembles that of fully connected neural networks, RaNN differs fundamentally in training: the hidden layer weights are randomly sampled from a prescribed distribution and fixed, while the output layer is trained analytically. This eliminates the need for backpropagation, offering substantial computational advantages. It has been shown that RaNN retains strong approximation and generalization properties under suitable conditions~\cite{Igelnik1995RNN1,Huang2006ELMtheorandapp,Liu2014ELMfeasible,Neufeld2023rnn}.

RaNN-based methods have already demonstrated success across a wide range of volumetric PDEs, including both strong-form (\cite{Dwivedi2020elm,Dong2021locELM,Chen2022rfm,AGRANN2025}) and weak-form (\cite{Shang2022DeepPetrov,Sun2024lrnndg,Dang2024hdg}) formulations. Applications include diffusive-viscous wave equations (\cite{Sun2024lrnndgdvwe}), KdV and Burgers equations (\cite{Sun2025lrnndgkdv}), linear elasticity and fluid dynamics (\cite{Shang2024DeepPetrov}), as well as more complex problems such as high-dimensional (\cite{Wang2023lelm}), interface (\cite{Chi2024rfm,Li2025int}), obstacle (\cite{Wang2024ob}), multiscale (\cite{Chen2024Micro-Macro,Linghu2024Homs-Rnn}), and inverse problems (\cite{Dong2023locELM}). Motivated by these developments, we extend RaNN methods to surface PDEs, focusing on both static and evolving geometries. We demonstrate the effectiveness and efficiency of the proposed approach through a series of numerical experiments.

The remainder of this paper is organized as follows. Section~\ref{sec:RaNN} reviews the architecture and training of Randomized Neural Networks. Section~\ref{sec:static_surface} presents RaNN formulations for both stationary and time-dependent PDEs on static surfaces, and provides a corresponding theoretical analysis for the parametrization-based formulation with interface compatibility. Section~\ref{sec:evolving} develops the method for evolving surfaces with topological invariance. Section~\ref{sec:ex} presents numerical results. Finally, Section~\ref{sec:conclusion} concludes the paper and outlines future directions.

%%%%%%%%%%%%%%%%%%%%%%%%%%%%%

\section{Randomized Neural Networks}
\label{sec:RaNN}

In this section, we introduce a special type of neural network, the \emph{Randomized Neural Network} (RaNN), which can be readily adapted to a wide range of PDEs while maintaining both high efficiency and accuracy.

\begin{figure}[H]  
	\centering   	
	\includegraphics[width=0.6\textwidth]{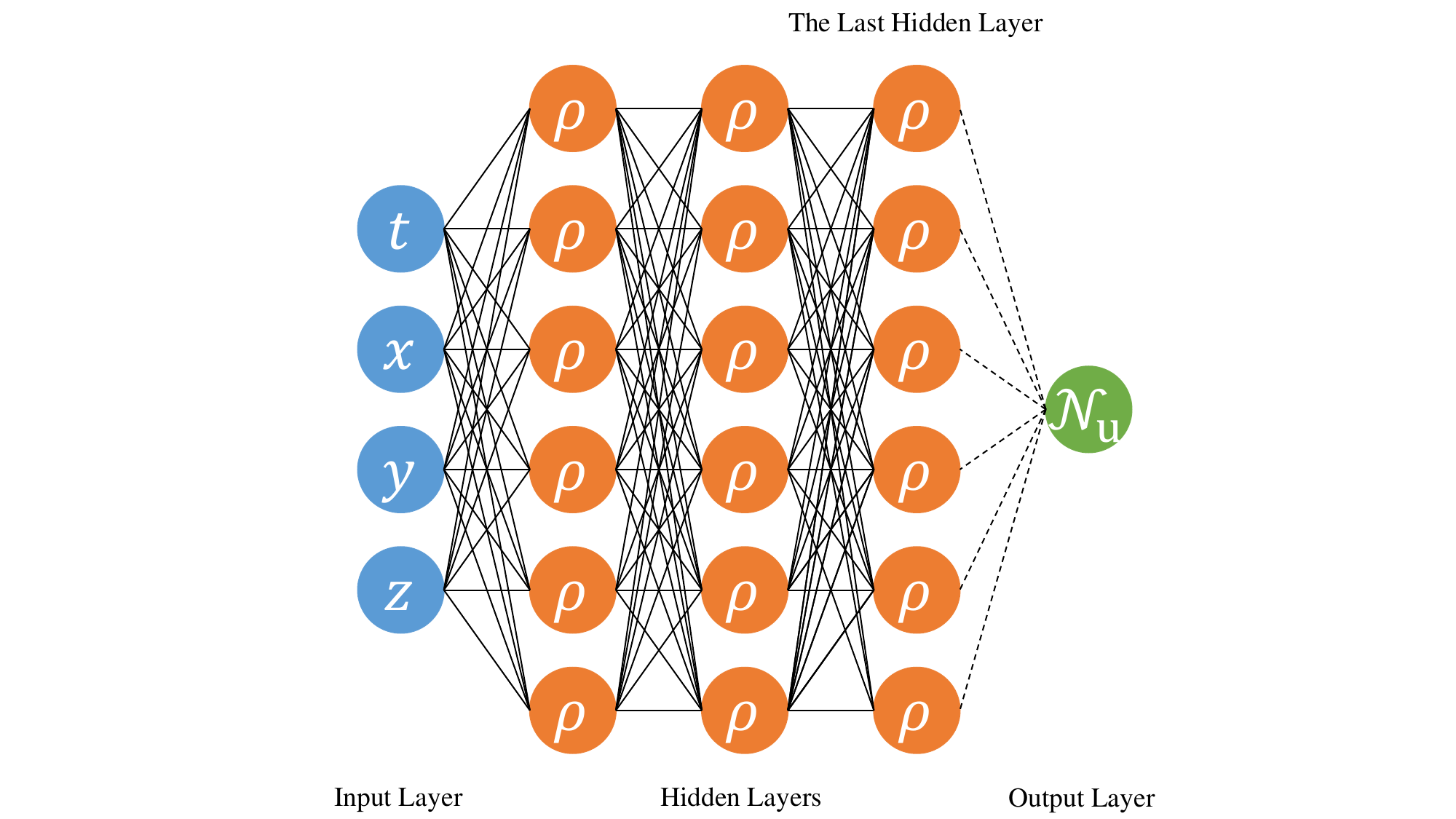}   
	\caption{Architecture of Randomized Neural Networks} 
	\label{nn_structure_surface}
\end{figure}

The architecture of a RaNN resembles that of a standard fully connected neural network. As an example, Figure~\ref{nn_structure_surface} illustrates a RaNN with a 4-dimensional input and a 1-dimensional output. The main distinction lies in how its parameters are handled. The weights connecting the input layer to the hidden layers, as well as those between hidden layers, are randomly generated according to a prescribed probability distribution and remain fixed throughout training. These connections are indicated by the solid black lines in the figure. In contrast, the weights connecting the final hidden layer to the output layer (dashed lines) are obtained via an analytical training procedure.

The activation function {$\rho$} is chosen to be nonlinear and non-polynomial. Let $M$ denote the number of neurons in the last hidden layer. The network output $u$ is given by
\begin{align}
\mathcal{N}_u(t,x,y,z) = \sum_{i=1}^{M} c_i \psi_i(\theta_i, t,x,y,z),
\end{align}
where $\psi_i(\theta_i,t,x,y,z)$ denotes the $i$-th output of the final hidden layer, and the hidden-layer parameters ${\theta_i}$ are randomly generated and then held fixed. The vector $\bc = (c_1, \dots, c_{M}) \in \mathbb{R}^{1\times M}$ contains the output-layer weights, which are the only trainable parameters.

While deep neural networks are known for their remarkable approximation capabilities, their application to partial differential equations typically requires solving a nonlinear and nonconvex optimization problem, often leading to significant optimization error. RaNN addresses this issue by making a modest compromise in approximation power in exchange for a reduction in optimization complexity. This yields a favorable balance between approximation and optimization errors, thereby improving accuracy while substantially reducing computational cost.

\section{RaNN methods for PDEs on static surfaces}
\label{sec:static_surface}

In this section, we employ the RaNN approach to solve PDEs on static surfaces and examine three settings: (i) surfaces described by a set of local parameterizations, (ii) surfaces defined implicitly via a zero-level set function, and (iii) surfaces given only through a finite set of sampled point coordinates. For each case, we develop the corresponding numerical method for solving surface PDEs.

Let $\Gamma$ be a two-dimensional static surface embedded in $\mathbb{R}^3$. A general form of the partial differential equation on the surface $\Gamma$ can be written as:
\begin{align}
\mathcal{L}(u) = f \quad \on \quad \Gamma,
\end{align}
where $u$ is the unknown function defined on the surface, $f$ is a given source term, and $\mathcal{L}$ is a differential operator that generally depends on the geometry of $\Gamma$. We assume $\Gamma$ is a closed surface; otherwise, appropriate boundary conditions must be imposed.

\subsection{Stationary partial differential equations}

We now consider a stationary PDE on a smooth, closed, connected surface $\Gamma$. For illustration, we take the Laplace--Beltrami equation:
\begin{align}
-\Delta_\Gamma u = f \quad \on \quad \Gamma,
\end{align}
where $u$ is the unknown function defined on $\Gamma$, and $f$ is the given source term.  On a closed surface, solvability requires the compatibility condition $\int_\Gamma f\,ds=0$,
and the solution is unique only up to an additive constant. In computations, one may fix the constant mode
either by imposing $\int_\Gamma u\,ds=0$ or by pinning $u(x_\ast)=0$ at a reference point.

\subsubsection{Parametrization-based approaches}\label{subsec:param-approaches}

We restrict attention to \emph{closed} surfaces~$\Gamma\subset\mathbb R^3$.
Assume that $\Gamma$ is described by a finite atlas of $C^k$ parametrizations
\[
  X_i:D_i\to \Gamma_i\subset\Gamma,\qquad i=1,\dots,J,
\]
where each $D_i\subset\R^2$ has $C^{1,1}$ boundary (or is piecewise $C^{1,1}$), this assumption is satisfied in the common case where $D_i$ is a rectangle.
Each $X_i$ is a $C^k$ immersion and admits a continuous extension
$\overline{X}_i:\overline{D_i}\to\overline{\Gamma_i}$ that is a homeomorphism onto its image
$\Gamma_i:=X_i(D_i)$.
The patches form a \emph{patch decomposition}:
\begin{equation}\label{eq:patch-decomp}
  \Gamma = \bigcup_{i=1}^J \Gamma_i,\qquad
  \mathrm{int}(\Gamma_i)\cap \mathrm{int}(\Gamma_j)=\emptyset\ (i\neq j),
\end{equation}
and neighboring patches may meet along smooth interface curves of \emph{positive length}
(positive one-dimensional Hausdorff measure). This setting is standard for parametric surface meshes
and avoids two-dimensional overlaps.

Following \cite{Dziuk2013fem}, we recall how to compute the surface metric tensor and the corresponding
local Laplace--Beltrami operator in each chart.
Let $\bxi=(\xi_1,\xi_2)^\top\in D_i$ and set $X_{i,\alpha}:=\partial_{\xi_\alpha}X_i$ ($\alpha\in\{1,2\}$).
Define the metric tensor $g^{(i)}_{\alpha\beta}:=X_{i,\alpha}\cdot X_{i,\beta}$ and
$g_i:=\det(g^{(i)}_{\alpha\beta})>0$, with inverse $(g_i^{\alpha\beta})=(g^{(i)}_{\alpha\beta})^{-1}$.
The surface measure on $\Gamma_i$ pulls back as
\[
  \dd s = \sqrt{g_i(\bxi)}\,\dd \bxi,\qquad \dd\bxi=\dd\xi_1\,\dd\xi_2,\qquad \bxi\in D_i.
\]
For a surface function $u$ define its local pullback $\tilde u_i:=u\circ X_i$ on $D_i$.
We define the chart operator $\Gop_i$ so that it represents the pullback of $-\Delta_\Gamma$, namely
\begin{equation}\label{eq:opGi}
  \Gop_i \tilde u_i := -\frac{1}{\sqrt{g_i}}\partial_{\xi_\alpha}\!\left(\sqrt{g_i}\,g_i^{\alpha\beta}\,\partial_{\xi_\beta} \tilde u_i\right),
\end{equation}
we adopt the Einstein summation convention over repeated Greek indices $\alpha,\beta\in\{1,2\}$, so that $(-\Delta_\Gamma u)\circ X_i = \Gop_i (u\circ X_i)$ in $D_i$.

Let $\rho:\mathbb{R}\to\mathbb{R}$ be an activation function. For each patch $i$, draw i.i.d.\ random weights
$\bw_i = [w_{i,1}, \cdots, w_{i,M_i}]^\top\in\mathbb{R}^{M_i\times 2}$ from the uniform distribution $U(-r_x, r_x)$, and choose
$\bB_i\in\mathbb{R}^{2\times M_i}$ according to the parameter ranges of $D_i$.
Define the bias vector $\bb_i\in\mathbb{R}^{M_i\times 1}$ by
\[
  \bb_i := -(\bw_i \odot \bB_i^{\top})\,\mathbf{1}_2,
  \qquad \mathbf{1}_2=(1,1)^\top,
\]
where $\odot$ denotes the Hadamard product. The random-feature basis on $D_i$ is
\[
  \bpsi_i(\bxi):=\rho(\bw_i\,\bxi+\bb_i)\in\mathbb{R}^{M_i \times 1},\qquad \bxi\in D_i,
\]
where $\rho(\cdot)$ is applied componentwise.
Define the patchwise trial class
\[
  \mathcal N^{\mathrm{par}}_i := \Big\{v_i = \bc_{i}\bpsi_{i}:\ \bc_i\in\R^{1\times M_i}\Big\},
\qquad
  \bm{\mathcal N}^{\mathrm{par}}_{\bm M}:=\prod_{i=1}^J \mathcal N^{\mathrm{par}}_{i},
\quad \bm M:=(M_1,\dots,M_J).
\]

Since the patches are non-overlapping, inter-patch coupling occurs only through their common boundary curves,
which we refer to as \emph{interfaces}.
Let $\mathcal E$ be the set of unoriented interfaces $e\subset \partial\Gamma_i\cap\partial\Gamma_j$ with $i\neq j$
and positive length.
For $e\in\mathcal E$ shared by patches $(i,j)$, define the parameter-edge preimages
\[
  \partial_e D_i:=X_i^{-1}(e)\subset \partial D_i,
  \qquad
  \partial_e D_j:=X_j^{-1}(e)\subset \partial D_j .
\]
Along $e$, the two parameter edges are naturally identified by the transition map
\[
  \phi_{ji}^e := X_j^{-1}\circ X_i\big|_{\partial_e D_i}:\ \partial_e D_i \to \partial_e D_j,
\]
so that for all $\bxi\in\partial_e D_i$,
\begin{equation}\label{eq:edge-id}
  X_i(\bxi)=X_j\big(\phi_{ji}^e(\bxi)\big)\in e .
\end{equation}
To compare traces across the interface, assume that $\phi_{ji}^e$ admits a $C^{k-1}$ extension to a neighborhood:
there exist open neighborhoods $\mathcal N_{i,e}\subset\overline{D_i}$ and $\mathcal N_{j,e}\subset\overline{D_j}$ and
a $C^{k-1}$ diffeomorphism
\[
  \Phi_{ji}^e:\mathcal N_{i,e}\to \mathcal N_{j,e},
  \qquad \Phi_{ji}^e\big|_{\partial_e D_i}=\phi_{ji}^e,
\]
with inverse $\Phi_{ij}^e=(\Phi_{ji}^e)^{-1}$ and $\Phi_{ji}^e(\partial_e D_i)=\partial_e D_j$.

Given $\bm v=(v_i)_{i=1}^J$, for each $e\in\mathcal E$ shared by $(i,j)$ we measure mismatches on the chosen side
$\partial_e D_i$. Let $\nu_{i,e}(\bxi)$ denote the outward unit normal (in $\mathbb R^2$) to the boundary curve
$\partial_e D_i\subset\partial D_i$, and define the normal derivative by $\partial_{\nu_{i,e}} v:=\nu_{i,e}\cdot\nabla v$.

\smallskip
\noindent
\emph{(i) Value mismatch.}
\[
  \delta_{e}^{0}(\bm v)(\bxi) := v_i(\bxi)-v_j(\Phi_{ji}^e(\bxi)),\qquad \bxi\in \partial_e D_i.
\]

\smallskip
\noindent
\emph{(ii) Normal-derivative mismatch (chain-rule consistent).}
\[
  \delta_{e}^{n}(\bm v)(\bxi)
  := \partial_{\nu_{i,e}} v_i(\bxi) - \partial_{\nu_{i,e}}\big(v_j\circ \Phi_{ji}^e\big)(\bxi),
  \qquad \bxi\in \partial_e D_i.
\]

Collect the two mismatches into
\[
  \Bop_e \bm v := \big(\delta_e^0(\bm v),\ \delta_e^{n}(\bm v)\big),
\qquad
  \Bop \bm v := (\Bop_e\bm v)_{e\in\mathcal E}.
\]

\begin{remark}[Why value and normal derivative are sufficient]\label{rem:val-normal-sufficient}
For $v\in H^2(D)$, the trace satisfies $v|_{\partial D}\in H^{3/2}(\partial D)$ and its tangential derivative obeys
$\partial_\tau v|_{\partial D}=\partial_\tau(v|_{\partial D})\in H^{1/2}(\partial D)$.
Hence, if $v_i=v_j\circ\Phi_{ji}^e$ on $\partial_e D_i$ in $H^{3/2}$, then tangential derivatives match automatically in $H^{1/2}$.
Therefore enforcing the normal derivative match in addition to the value match suffices for global $H^2$-gluing across interfaces.
\end{remark}

\noindent\textbf{Parametrization-based approach.} Let $\bu=(u_i)_{i=1}^J \in \bm{\mathcal N}^{\mathrm{par}}_{\bm M}$ denote the atlas-based approximation, where each
$u_i$ is represented by a local RaNN on $D_i$.
By pulling back the Laplace--Beltrami equation on $\Gamma$ to each chart, we obtain the coupled problem on
$\mathcal D:=\{D_i\}_{i=1}^J$:
\begin{subequations}\label{m_ti_par_bc}
\begin{align}
\Gop_i u_i &= f\circ X_i \quad \text{in } D_i,\quad i=1,\dots,J,\\
\Bop\,\bu &= 0 \quad\ \ \ \ \ \text{on } \mathcal E .
\end{align}
\end{subequations}
Enforcing \eqref{m_ti_par_bc} at collocation points (interior points for the PDE and interface points for $\Bop$)
yields a linear system for the output weights $\bc=\{\bc_i\}_{i=1}^J$, which is solved in the least-squares sense.

\noindent\textbf{Parametrization-based approach(no separate boundary enforcement)}. We define a \emph{single} randomized neural network on embedded coordinates $x\in\Gamma\subset\mathbb R^3$:
\[
  \hat u(x):=\bc\,\rho(\bw x+\bb),\qquad x\in\Gamma,
\]
where $\bw\in\mathbb R^{M\times 3}$ has i.i.d.\ entries distributed as $U(-r_x,r_x)$, $\bc\in\mathbb R^{1\times M}$
is trainable, and
\[
  \bb:=-(\bw\odot \bB^{\!\top})\,\mathbf 1_3,
  \qquad \mathbf 1_3=(1,1,1)^{\top},
\]
with $\bB=[\bB_1,\dots,\bB_{M}]\in\mathbb R^{3\times M}$ chosen to cover the coordinate ranges of $\Gamma$
(e.g., by sampling $\bB_m$ i.i.d.\ from a uniform distribution supported on a bounding box of~$\Gamma$).
On each patch we work with the pullback
\[
  u_i(\bxi):=\hat u(X_i(\bxi))=\bc\,\rho\!\left(\bw X_i(\bxi)+\bb\right),\qquad \bxi\in D_i.
\]

Since $\hat u$ is globally defined on $\Gamma$, let $e=\Gamma_i\cap\Gamma_j$ and $\bxi\in\partial_e D_i$.
By the interface identification $X_i(\bxi)=X_j(\Phi_{ji}^e(\bxi))$, we have
\[
u_i(\bxi)=\hat u(X_i(\bxi))=\hat u(X_j(\Phi_{ji}^e(\bxi)))=u_j(\Phi_{ji}^e(\bxi)),
\]
so $\delta_e^0(\bu)=0$.

Moreover, let $\nu_{i,e}(\bxi)$ be the outward unit normal to $\partial_e D_i\subset\partial D_i$ in $\mathbb R^2$.
Since $u_i=\hat u\circ X_i$ and $u_j\circ \Phi_{ji}^e=\hat u\circ X_j\circ \Phi_{ji}^e=\hat u\circ X_i$ on a neighborhood
, differentiating in the normal direction
and applying the chain rule yields
\[
\partial_{\nu_{i,e}}u_i(\bxi)=\partial_{\nu_{i,e}}\big(u_j\circ \Phi_{ji}^e\big)(\bxi),
\]
hence $\delta_e^{n}(\bu)=0$ and therefore $\Bop\,\bu=0$ holds automatically. Consequently, no separate interface mismatch
enforcement is required for this globally-defined ansatz.

This argument relies on the existence of a $C^{1}$ extension of the transition map in a neighborhood of each interface,
as assumed later in the geometric setting of Section~\ref{subsec:param-analysis}.
Using $\Gop_i$ in \eqref{eq:opGi}, we solve the atlas-form PDE: find $\hat u$ such that, for each $i=1,\dots,J$,
\begin{equation}\label{no_boundary}
  \Gop_i u_i(\bxi)= (f\circ X_i)(\bxi)\quad\text{in }D_i.
\end{equation}
Collocation on the interiors of $\{D_i\}$ yields a linear least-squares system for the shared output weights~$\bc$.

\begin{remark}
The choice of $\bB=[\bB_1,\ldots,\bB_{M}]\in\mathbb{R}^{d\times M}$ imposes no additional computational burden, since it is determined from the surface characteristics.  Let $\tilde D\subset\mathbb{R}^d$ denote the axis-aligned bounding box associated with the chosen input space: 
if $d=3$, $\tilde D=[x_{\min},x_{\max}]\times[y_{\min},y_{\max}]\times[z_{\min},z_{\max}]$, where $[x,y,z]_{\min/\max}=\inf/\sup_{(\xi_1,\xi_2)\in D}X(\xi_1,\xi_2)$, etc.; 
if $d=2$, $\tilde D=[\min(\xi_1),\max(\xi_1)]\times[\min(\xi_2),\max(\xi_2)]$. 
We then draw the columns i.i.d. as $\bB_i\sim U(\tilde D)$.

The above construction of $\bb$ (via $\bB$) is adopted only in our numerical experiments to conveniently cover the input domain.
In the subsequent theoretical analysis, we instead assume that the inner parameters $\bw$ and $\bb$ are sampled independently
from prescribed uniform distributions.
\end{remark}

\begin{remark}[Fixing the additive constant on closed surfaces]\label{rem:const-shift}
Let $\Gamma$ be a closed surface. Since $\ker(\Delta_\Gamma)=\mathrm{span}\{1\}$, the surface Poisson problem
$-\Delta_\Gamma u=f$ (with $\int_\Gamma f\,\dd s=0$) admits solutions that are unique only up to an additive constant.
Accordingly, we interpret $u$ in the quotient space $H^1(\Gamma)/\mathbb R$.

In the remainder of this paper, the convergence analysis is carried out in the mean-zero subspace
(equivalently, in $H^1(\Gamma)/\mathbb R$) by selecting the unique representative satisfying
\[
  \int_\Gamma u\,\dd s=0.
\]
All error estimates are understood for these normalized representatives, i.e., modulo additive constants.

In numerical experiments, we do not enforce the mean-zero constraint explicitly. Instead, we remove the constant mode by
aligning the numerical approximation with the exact solution at a fixed reference point $x_\ast\in\Gamma$:
\[
  \tilde u_h(x):=u_h(x)-u_h(x_\ast)+u(x_\ast).
\]
We then compute errors using $\tilde u_h$. This post-processing differs from the mean-zero normalization only by a constant
and therefore does not affect quantities that are invariant under additive constants, while providing a consistent
representative for $L^2$-type error measurements.
\end{remark}

\subsubsection{Theoretical analysis of the parametrization-based approach}\label{subsec:param-analysis}

We develop a theoretical analysis for the parametrization-based approach with explicit interface
mismatch penalization in a residual-based learning framework using RaNN.
The idea is to decompose the total error into
approximation, statistical, and optimization components in an abstract Hilbert-space setting.
Throughout, we focus on the Laplace--Beltrami equation on a closed, connected surface for concreteness; see
Remark~\ref{rem:general-elliptic} for extensions to more general linear surface PDEs.

\smallskip
\noindent\textbf{Standing solvability.}
Since $\Gamma$ is closed and $\ker(\Delta_\Gamma)=\mathrm{span}\{1\}$, we assume the compatibility condition
$\int_\Gamma f\,\dd s=0$ and select the unique mean-zero solution $u^\ast\in H^2(\Gamma)$ satisfying
$\int_\Gamma u^\ast\,\dd s=0$.

\smallskip
\noindent\textbf{Geometric setting.}
We work with the patch decomposition \eqref{eq:patch-decomp} and the chart operators $\{\Gop_i\}_{i=1}^J$ in
\eqref{eq:opGi}.
To give $\Bop$ a coordinate-invariant meaning while keeping a non-overlapping decomposition, we assume that each
$X_i:D_i\to\Gamma_i$ is the restriction of a $C^k$ chart $\widetilde X_i:\widetilde D_i\to \Gamma$ defined on an open
set $\widetilde D_i\supset\overline{D_i}$.
For each interface $e=\Gamma_i\cap\Gamma_j$ there exist open neighborhoods
$\mathcal N_{i,e}\subset \widetilde D_i$ and $\mathcal N_{j,e}\subset \widetilde D_j$ and a $C^{k-1}$ diffeomorphism
$\Phi_{ji}^e:\mathcal N_{i,e}\to\mathcal N_{j,e}$ such that
$\widetilde X_j\circ \Phi_{ji}^e=\widetilde X_i$ on $\mathcal N_{i,e}$, and
$\Phi_{ji}^e|_{\partial_e D_i}=X_j^{-1}\circ X_i|_{\partial_e D_i}$.
This is the standard atlas compatibility assumption on smooth manifolds and does not introduce two-dimensional
overlaps in the \emph{computational} patch decomposition \eqref{eq:patch-decomp}.

\smallskip
\noindent\textbf{Function spaces.}
Let $H^m(D_i)$ be the usual Sobolev space and define the broken space
\[
  \bm H^m(\mathcal D):=\prod_{i=1}^J H^m(D_i),\qquad \bm v=(v_1,\dots,v_J).
\]
We use the consistent surface integral for $\bm v\in\bm H^0(\mathcal D)$,
\[
  \int_\Gamma v\,\dd s := \sum_{i=1}^J \int_{D_i} v_i(\bxi)\,\sqrt{g_i(\bxi)}\,\dd \bxi,
\]
and define the global mean-zero projector
\begin{equation}\label{eq:Pi-diam-atlas}
  (\Pi_{\diam}\bm v)_i := v_i - \frac{\sum_{m=1}^J \int_{D_m} v_m\sqrt{g_m} \, \dd \bxi}{\sum_{m=1}^J \int_{D_m}\sqrt{g_m}\, \dd \bxi},
  \qquad i=1,\dots,J.
\end{equation}
Since constants have zero mismatches across interfaces and $\Gop_i(\mathrm{const})=0$, $\Pi_\diam$ preserves the interface
constraints and leaves both population and empirical losses unchanged.

\noindent\textbf{Interface mismatch norms and trace spaces.}
For each interface $e\in\mathcal E$, fix an orientation $i(e):=\min\{i,j\}$ and set
$\partial_e D:=\partial_e D_{i(e)}\subset\partial D_{i(e)}$, equipped with the arc-length measure $\dd l$.
Define the computable normed space
\[
  Z_e := \big(L^2(\partial_e D)\big)^2,\qquad
  \|\Bop_e\bm v\|_{Z_e}^2 := \|\delta_e^0(\bm v)\|_{L^2(\partial_e D)}^2 + \|\delta_e^{n}(\bm v)\|_{L^2(\partial_e D)}^2,
\]
and $Z:=\prod_{e\in\mathcal E} Z_e$ with $\|\Bop\bm v\|_Z^2:=\sum_{e\in\mathcal E}\|\Bop_e\bm v\|_{Z_e}^2$.
For stability we also use the trace-scale space
\[
  \mathcal Z_e := H^{3/2}(\partial_e D)\times H^{1/2}(\partial_e D),\qquad
  \|\Bop_e\bm v\|_{\mathcal Z_e}^2
  := \|\delta_e^0(\bm v)\|_{H^{3/2}(\partial_e D)}^2 + \|\delta_e^{n}(\bm v)\|_{H^{1/2}(\partial_e D)}^2,
\]
and $\mathcal Z:=\prod_{e\in\mathcal E}\mathcal Z_e$ with $\|\Bop\bm v\|_{\mathcal Z}^2:=\sum_{e\in\mathcal E}\|\Bop_e\bm v\|_{\mathcal Z_e}^2$.
Note that $\Bop\bm v=0$ is understood in this trace sense.

\medskip
\noindent\textbf{Population and empirical residual losses.}
Let $\tilde f_i:=f\circ X_i$. For $\bm v\in \bm H^2(\mathcal D)$ define
\[
  \mathcal J(\bm v):=
  \sum_{i=1}^J \|\Gop_i v_i-\tilde f_i\|_{L^2(D_i)}^2
  + \eta \|\Bop \bm v\|_{Z}^2.
\]
For the empirical loss, let $\{\bxi_{i,n}\}_{n=1}^{N_i}\overset{\text{i.i.d.}}{\sim}U(D_i)$ and
let $\{\bxi_{e,n}\}_{n=1}^{N_e}\overset{\text{i.i.d.}}{\sim}U(\partial_e D)$ be i.i.d.\ samples with respect to the
normalized arc-length measure on $\partial_e D$. Then
\[
  \widehat{\mathcal J}(\bm v):=
  \sum_{i=1}^J \frac{|D_i|}{N_i}\sum_{n=1}^{N_i}\big|\Gop_i v_i(\bxi_{i,n})-\tilde f_i(\bxi_{i,n})\big|^2
  +\eta \sum_{e\in\mathcal E}\frac{|\partial_e D|}{N_e}\sum_{n=1}^{N_e}\|\Bop_e \bm v(\bxi_{e,n})\|_{\R^2}^2.
\]

\smallskip
\noindent\textbf{Trial function class.}
Write the patchwise RaNN output as
\[
  v_i(\bxi)=\sum_{m=1}^{M_i} c_{i,m}\,\psi_{i,m}(\bxi),\qquad \bxi\in D_i,
\]
where $\psi_{i,m}$ are the fixed random features generated in the previous subsection and $\bc_i=(c_{i,1},\dots,c_{i,M_i})$.
For the statistical analysis we work with an $\ell^1$-bounded coefficient set:
\[
  \mathcal N^{par}_{i,C_N}:=\Big\{v_i=\sum_{m=1}^{M_i}c_{i,m}\psi_{i,m}:\ \|c_i\|_1\le C_N\Big\},\qquad
  \bm{\mathcal N}^{par}_{\bm M,C_N}:=\prod_{i=1}^J \mathcal N^{par}_{i,C_N}.
\]
We impose the global mean-zero constraint through the projector:
\[
  \Pi_\diam\bm{\mathcal N}^{par}_{\bm M,C_N}:=\{\Pi_\diam \bm v:\ \bm v\in \bm{\mathcal N}^{par}_{\bm M,C_N}\}.
\]
The empirical (least-squares) minimizer is
\begin{equation}\label{eq:uM-A-atlas}
  \bm u_{\bm M} \in \arg\min_{\bm v\in \Pi_\diam\bm{\mathcal N}^{par}_{\bm M,C_N}} \widehat{\mathcal J}(\bm v).
\end{equation}
We measure errors in the broken norm
\[
  \|\bm v\|_{\bm H^2(\mathcal D)}^2:=\sum_{i=1}^J \|v_i\|_{H^2(D_i)}^2.
\]

\medskip
\noindent\textbf{Step 1: analytic stability (graph-norm control).}
We first establish that the residual operator controls the broken $H^2$ norm (up to constants), which is the key
ingredient turning residual minimization into an error estimate.

To connect the residual-based loss to the actual error on the surface, we need a stability mechanism that converts
(i) patchwise PDE residuals measured in $L^2(D_i)$ and (ii) inter-patch incompatibilities measured on interfaces,
into a control of the broken $H^2$-error in $\bm H^2(\mathcal D)$.
The argument proceeds in three conceptual steps:
(1) uniform ellipticity and coefficient regularity of the local operators $\{\Gop_i\}$ on each chart, guaranteed by the metric properties of the parametrizations;
(2) a global graph-norm estimate on the closed surface, which controls the mean-zero $H^2(\Gamma)$ norm by the $L^2(\Gamma)$ norm of $\Delta_\Gamma u$;
and (3) a gluing--lifting construction on the atlas: zero interface mismatches allow us to glue a broken function into a global $H^2(\Gamma)$ function,
while a bounded right-inverse of the mismatch trace operator lifts general interface defects into a corrector.
Together, these ingredients yield a graph-norm inequality on the atlas, which is the key stability estimate underlying the subsequent
approximation--statistical--optimization error decomposition.

\begin{lemma}[Uniform ellipticity on each chart]\label{lem:unif-ellip-atlas}
For each $i$ there exist $0<\lambda_i\le \Lambda_i<\infty$ such that
\[
  \lambda_i|\zeta|^2 \le g_i^{\alpha\beta}(\bxi)\zeta_\alpha\zeta_\beta \le \Lambda_i|\zeta|^2,
  \qquad \forall \zeta\in\R^2,\ \forall \bxi\in D_i,
\]
and $g_i^{\alpha\beta},\sqrt{g_i}\in C^{k-1}(\overline{D_i})$.
\end{lemma}

\begin{proof}
Since $X_i\in C^k(\overline{D_i})$ is an immersion, the Jacobian $J_i=[X_{i,1}\ X_{i,2}]$ has rank~$2$ on $\overline{D_i}$.
Hence the metric tensor $g^{(i)}(\bxi)=J_i(\bxi)^\top J_i(\bxi)$ is symmetric positive definite for every $\bxi\in\overline{D_i}$.
Let $\lambda_{\min}(\bxi)$ and $\lambda_{\max}(\bxi)$ denote the extreme eigenvalues of $g^{(i)}(\bxi)$.
By continuity of $J_i$ (hence of $g^{(i)}$) and compactness of $\overline{D_i}$,
$\lambda_{\min}$ and $\lambda_{\max}$ attain their extrema and satisfy
$0<\underline\lambda:=\min_{\overline{D_i}}\lambda_{\min}\le \max_{\overline{D_i}}\lambda_{\max}=: \overline\lambda<\infty$.
Therefore, for $(g_i^{\alpha\beta})=(g^{(i)}_{\alpha\beta})^{-1}$,
\[
  \frac1{\overline\lambda}|\zeta|^2 \le g_i^{\alpha\beta}(\bxi)\zeta_\alpha\zeta_\beta \le \frac1{\underline\lambda}|\zeta|^2,
  \qquad \forall \zeta\in\R^2,\ \forall \bxi\in D_i,
\]
so we may take $\lambda_i=1/\overline\lambda$ and $\Lambda_i=1/\underline\lambda$.
Moreover, $g^{(i)}_{\alpha\beta}=X_{i,\alpha}\cdot X_{i,\beta}\in C^{k-1}(\overline{D_i})$, hence
$g_i=\det(g^{(i)}_{\alpha\beta})\in C^{k-1}(\overline{D_i})$ and $g_i>0$ implies $\sqrt{g_i}\in C^{k-1}(\overline{D_i})$.
Since inversion is smooth on SPD matrices, $g_i^{\alpha\beta}\in C^{k-1}(\overline{D_i})$ as well.
\end{proof}

\begin{theorem}[Elliptic regularity on a closed connected surface (graph norm)]\label{thm:graph-surface-H2}
Let $\Gamma\subset\mathbb R^3$ be a closed, connected $C^{k}$ surface with $k\ge 2$.
Then there exists $C_\Gamma>0$ such that for all $u\in H^2(\Gamma)$ with $\int_\Gamma u\,\dd s=0$,
\[
  \|u\|_{H^2(\Gamma)} \le C_\Gamma\,\|\Delta_\Gamma u\|_{L^2(\Gamma)}.
\]
\end{theorem}

\begin{proof}
On the compact $C^k$ manifold $\Gamma$ without boundary, standard elliptic regularity for the strongly elliptic operator
$\Delta_\Gamma$ yields
\[
  \|u\|_{H^2(\Gamma)}\le C_1\big(\|\Delta_\Gamma u\|_{L^2(\Gamma)}+\|u\|_{L^2(\Gamma)}\big),
  \qquad \forall u\in H^2(\Gamma),
\]
see, e.g., \cite{Hebey2000manifolds}.
To remove the $\|u\|_{L^2}$ term, use that $\Gamma$ is connected and $u$ has zero mean, so the surface Poincar\'e inequality gives
$\|u\|_{L^2(\Gamma)}\le C_P\|\nabla_\Gamma u\|_{L^2(\Gamma)}$.
Moreover, since $\Gamma$ has no boundary,
\[
  \|\nabla_\Gamma u\|_{L^2(\Gamma)}^2
  =-\int_\Gamma u\,\Delta_\Gamma u\,\dd s
  \le \|u\|_{L^2(\Gamma)}\,\|\Delta_\Gamma u\|_{L^2(\Gamma)}.
\]
Combining the two inequalities yields $\|u\|_{L^2(\Gamma)}\le C\,\|\Delta_\Gamma u\|_{L^2(\Gamma)}$, and substituting back gives the claim.
\end{proof}

We now bridge the gap between the broken (patchwise) Sobolev space on the parameter domains
and the global Sobolev space on the surface. The key idea is standard:
(i) if the interface mismatches vanish, a broken $H^2$ field can be glued into a global $H^2(\Gamma)$ function;
(ii) if mismatches do not vanish, they can be lifted into a bulk correction supported near interfaces.
These two steps allow us to transfer the surface elliptic regularity estimate in
Theorem~\ref{thm:graph-surface-H2} to the atlas formulation.

\begin{lemma}[Gluing: zero mismatches imply a global $H^2(\Gamma)$ function]\label{lem:glue}
Let $\bm v\in \bm H^2(\mathcal D)$ satisfy $\Bop \bm v=0$ (in the trace sense). Then there exists a unique $u\in H^2(\Gamma)$ such that $u\circ X_i=v_i$ a.e.\ in $D_i$ for all $i$.
Moreover, there exist atlas-dependent constants $C_{r},C_{l}>0$ such that
\[
  \|u\|_{H^2(\Gamma)} \le C_{r}\,\|\bm v\|_{\bm H^2(\mathcal D)},
  \qquad
  \|\bm v\|_{\bm H^2(\mathcal D)} \le C_{l}\,\|u\|_{H^2(\Gamma)}.
\]
\end{lemma}

\begin{proof}
Define $u$ patchwise by $u|_{\Gamma_i}:=v_i\circ X_i^{-1}$. Since each $X_i:D_i\to\Gamma_i$ is a $C^k$ diffeomorphism onto
its image, Sobolev pullback stability yields $u|_{\Gamma_i}\in H^2(\Gamma_i)$.

Because $v_i\in H^2(D_i)$, the boundary traces satisfy $v_i|_{\partial D_i}\in H^{3/2}(\partial D_i)$ and
$\partial_{\nu_{i}} v_i|_{\partial D_i}\in H^{1/2}(\partial D_i)$ (here $\nu_i$ is the outward unit normal on $\partial D_i$).
The condition $\Bop\bm v=0$ means that, on every interface $e=\Gamma_i\cap\Gamma_j$ identified via the transition map $\Phi_{ji}^e$,
both the value traces and the normal-derivative traces agree across the interface in the appropriate
$H^{3/2}/H^{1/2}$ trace sense. Hence $u$ is single-valued across interfaces and its co-normal derivatives match.

By standard broken-Sobolev gluing characterizations (piecewise $H^2$ plus continuity of $u$ and $\partial_\nu u$ across interfaces),
transported through the chart identifications, we conclude $u\in H^2(\Gamma)$; see, e.g., \cite{Brenner2008analysis}.

Finally, the norm equivalences follow from boundedness of the $H^2$ norms under the chart maps and the finiteness of the atlas:
\[
  \|u\|_{H^2(\Gamma)}^2 \simeq \sum_{i=1}^J \|u\|_{H^2(\Gamma_i)}^2
  \simeq \sum_{i=1}^J \|u\circ X_i\|_{H^2(D_i)}^2
  = \sum_{i=1}^J \|v_i\|_{H^2(D_i)}^2,
\]
where the hidden constants depend only on uniform bounds on $DX_i$, $D^2X_i$ and their inverses on $\overline{D_i}$.
\end{proof}

\begin{lemma}[Lifting of interface mismatch traces]\label{lem:lifting-atlas}
There exists a bounded linear operator
\[
  \mathcal R:\mathrm{Range}(\Bop)\subset \mathcal Z \to \bm H^2(\mathcal D)
\]
such that for every $h\in \mathrm{Range}(\Bop)$,
\[
  \Bop(\mathcal R h)=h,\qquad \int_\Gamma (\mathcal R h)\,\dd s=0,\qquad
  \|\mathcal R h\|_{\bm H^2(\mathcal D)} \le C_{\mathrm{lift}}\,\|h\|_{\mathcal Z},
\]
where $C_{\mathrm{lift}}$ depends only on the atlas. Moreover one may take
$C_{\mathrm{lift}}\lesssim \big(\max_{e\in\mathcal E}C_e\big)\sqrt{d_{\max}}$, where
$d_{\max}:=\max_{1\le i\le J}\#\{e\in\mathcal E:\ e\text{ is incident to }i\}$ and $C_e$ are the local extension constants
defined below.
\end{lemma}

\begin{proof}
\noindent\textbf{Step 1: local right-inverse for a restricted $H^2$ trace.}
Fix an interface $e\in\mathcal E$ and its chosen side $i(e)$. Consider the restricted trace operator on $\partial_e D_{i(e)}$,
\[
  \gamma_{i(e),e}:\ H^2(D_{i(e)})\to H^{3/2}(\partial_e D_{i(e)})\times H^{1/2}(\partial_e D_{i(e)}),\quad
  \gamma_{i(e),e}(v):=\big(v,\partial_{\nu_{i(e),e}}v\big)\big|_{\partial_e D_{i(e)}},
\]
where $\nu_{i(e),e}$ is the outward unit normal (in $\mathbb R^2$) along $\partial_e D_{i(e)}$.

Since $\partial D_{i(e)}$ is assumed to be piecewise $C^{1,1}$, the full-boundary trace map
\[
  \gamma_{i(e)}:\ H^2(D_{i(e)})\to H^{3/2}(\partial D_{i(e)})\times H^{1/2}(\partial D_{i(e)}),\quad
  \gamma_{i(e)}(v):=\big(v,\partial_{\nu_{i(e)}}v\big)\big|_{\partial D_{i(e)}},
\]
is surjective and admits a bounded right inverse
$R_{i(e)}:H^{3/2}(\partial D_{i(e)})\times H^{1/2}(\partial D_{i(e)})\to H^2(D_{i(e)})$
(see, e.g., \cite{McLean2000analysis} and standard localization arguments for piecewise $C^{1,1}$ boundaries).

Using boundary localization, we construct a bounded linear operator
\[
  E_e:\mathcal Z_e\to H^2(D_{i(e)}),\qquad \mathcal Z_e:=H^{3/2}(\partial_e D_{i(e)})\times H^{1/2}(\partial_e D_{i(e)}),
\]
such that
\begin{equation}\label{eq:local-right-inv-strong}
  \gamma_{i(e),e}(E_e h_e)=h_e \quad \text{on }\partial_e D_{i(e)}\ \text{(in the trace sense)},\qquad
  \|E_e h_e\|_{H^2(D_{i(e)})}\le C_e\,\|h_e\|_{\mathcal Z_e},
\end{equation}
where $C_e$ depends only on the atlas (in particular on the piecewise $C^{1,1}$ character of $\partial D_{i(e)}$ and the choice of the interface segment).

\smallskip
\noindent
\emph{One concrete construction.}
Let $h_e=(g_0,g_1)\in \mathcal Z_e$.
First extend $(g_0,g_1)$ from the boundary segment $\partial_e D_{i(e)}$ to the full boundary $\partial D_{i(e)}$
by a bounded linear extension operator
\[
  \mathcal E^{\partial}_e:\ H^{3/2}(\partial_e D_{i(e)})\times H^{1/2}(\partial_e D_{i(e)})
  \to H^{3/2}(\partial D_{i(e)})\times H^{1/2}(\partial D_{i(e)}),
\]
so that $\mathcal E^{\partial}_e(g_0,g_1)=(\tilde g_0,\tilde g_1)$ and $(\tilde g_0,\tilde g_1)=(g_0,g_1)$ on $\partial_e D_{i(e)}$.
Set $\tilde v:=R_{i(e)}(\tilde g_0,\tilde g_1)\in H^2(D_{i(e)})$, so that $\gamma_{i(e)}(\tilde v)=(\tilde g_0,\tilde g_1)$.

Next choose a smooth cutoff $\chi_e\in C^\infty(\overline{D_{i(e)}})$ supported in a collar neighborhood of $\partial_e D_{i(e)}$
such that $\partial_{\nu_{i(e),e}}\chi_e=0$ on $\partial_e D_{i(e)}$, and $\chi_e\equiv 1$ on $\partial_e D_{i(e)}$
except possibly in arbitrarily small neighborhoods of its two endpoints (which is irrelevant in the trace sense).
Moreover, we choose the collar neighborhood and the cutoff $\chi_e$ so that
$\operatorname{supp}(\chi_e)\cap \partial D_{i(e)}\subset \partial_e D_{i(e)}$.
Consequently, the trace of $E_e h_e=\chi_e\tilde v$ and its normal derivative trace vanish on
$\partial D_{i(e)}\setminus \partial_e D_{i(e)}$ in the trace sense, so different interface liftings do not interfere.
Define $E_e h_e:=\chi_e\,\tilde v$. Then on $\partial_e D_{i(e)}$ (in the trace sense),
\[
(E_e h_e)|_{\partial_e D_{i(e)}}=\tilde v|_{\partial_e D_{i(e)}}=g_0,\qquad
\partial_{\nu_{i(e),e}}(E_e h_e)|_{\partial_e D_{i(e)}}
=\chi_e\,\partial_{\nu_{i(e),e}}\tilde v + (\partial_{\nu_{i(e),e}}\chi_e)\tilde v
=\partial_{\nu_{i(e),e}}\tilde v
=g_1,
\]
so \eqref{eq:local-right-inv-strong} holds. Boundedness in $H^2(D_{i(e)})$ follows from product rules and boundedness of $\chi_e$
and its derivatives.

\medskip
\noindent\textbf{Step 2: assemble a global lifting on each patch.}
For $h=(h_e)_{e\in\mathcal E}\in\mathrm{Range}(\Bop)$ define, patchwise,
\[
  (\mathcal Rh)_i:=\sum_{e:\,i(e)=i}E_e h_e,\qquad i=1,\dots,J.
\]
Because each $E_e h_e$ is localized near $\partial_e D_{i(e)}$ and reproduces $(g_0,g_1)$ on that segment,
its contribution to $\gamma_{i(e),e}((\mathcal Rh)_{i(e)})$ is exactly $h_e$.
Therefore the above assembly is well-defined, and hence $\Bop(\mathcal Rh)=h$. 

\medskip
\noindent\textbf{Step 3: norm bound.}
Using Cauchy--Schwarz and that each patch is incident to at most $d_{\max}$ interfaces,
\[
  \|(\mathcal Rh)_i\|_{H^2(D_i)}^2
  =\Big\|\sum_{e:\,i(e)=i}E_e h_e\Big\|_{H^2(D_i)}^2
  \le d_{\max}\sum_{e:\,i(e)=i}\|E_e h_e\|_{H^2(D_i)}^2
  \le d_{\max}\Big(\max_{e}C_e^2\Big)\sum_{e:\,i(e)=i}\|h_e\|_{\mathcal Z_e}^2.
\]
Summing over $i$ yields
\[
  \|\mathcal Rh\|_{\bm H^2(\mathcal D)}\le (\max_e C_e)\sqrt{d_{\max}}\,\|h\|_{\mathcal Z}.
\]

\medskip
\noindent\textbf{Step 4: enforce mean-zero normalization.}
Let $c:=|\Gamma|^{-1}\int_\Gamma (\mathcal Rh)\,\dd s$ and subtract $c$ from every component. Since constants have zero
interface mismatches, $\Bop(\mathcal Rh)$ is unchanged while $\int_\Gamma (\mathcal Rh)\,\dd s=0$ holds.
\end{proof}

\begin{theorem}[Graph-norm control with interface mismatch traces]\label{thm:graphA-atlas-trace-eta}
Assume Lemma~\ref{lem:unif-ellip-atlas} and that $\Gamma$ is closed and connected. 
Let $\eta>0$. Then there exists a constant $C_{\mathrm{tr}}>0$, depending only on $\Gamma$, the atlas and $\eta$, such that
for every $\bm v\in \bm H^2(\mathcal D)$ with $\int_\Gamma v\,\dd s=0$,
\[
  \|\bm v\|_{\bm H^2(\mathcal D)}
  \le C_{\mathrm{tr}}
  \left(
     \Big(\sum_{i=1}^J \|\Gop_i v_i\|_{L^2(D_i)}^2\Big)^{1/2}
     + \sqrt{\eta}\,\|\Bop \bm v\|_{\mathcal Z}
  \right),
\]
Moreover, one can take
\[
  C_{\mathrm{tr}} := \max\!\left\{C_1,\ \frac{C_2}{\sqrt{\eta}}\right\},
\]
where $C_1,C_2>0$ depend only on $\Gamma$ and the atlas.
\end{theorem}

\begin{proof}
\noindent\textbf{Step 1: remove interface mismatches by lifting.}
Let $h:=\Bop \bm v\in \mathrm{Range}(\Bop)\subset \mathcal Z$ and set $\bm\psi:=\mathcal R h$ from Lemma~\ref{lem:lifting-atlas}.
Define $\bm w:=\bm v-\bm\psi$. Then $\Bop \bm w=0$. Moreover, since both $\int_\Gamma v\,\dd s=0$ and
$\int_\Gamma(\mathcal Rh)\,\dd s=0$, we also have $\int_\Gamma w\,\dd s=0$.

\medskip
\noindent\textbf{Step 2: glue $\bm w$ into a global surface function.}
Because $\Bop\bm w=0$ is understood in the natural trace spaces $H^{3/2}$ (values) and $H^{1/2}$ (normal derivatives),
Lemma~\ref{lem:glue} yields a unique $u\in H^2(\Gamma)$ such that $u\circ X_i=w_i$ a.e.\ in $D_i$ for all $i$.
Furthermore, by stability of Sobolev pullbacks under $C^k$ chart maps and finiteness of the atlas, there exists
an atlas-dependent constant $C_{l}>0$ such that
\begin{equation}\label{eq:w-to-u}
  \|\bm w\|_{\bm H^2(\mathcal D)} \le C_{l}\,\|u\|_{H^2(\Gamma)} .
\end{equation}

\medskip
\noindent\textbf{Step 3: apply surface elliptic regularity and transfer to charts.}
By Theorem~\ref{thm:graph-surface-H2} (mean-zero on a closed connected surface),
\[
  \|u\|_{H^2(\Gamma)} \le C_\Gamma\,\|\Delta_\Gamma u\|_{L^2(\Gamma)}.
\]
Using $(-\Delta_\Gamma u)\circ X_i=\Gop_i(u\circ X_i)=\Gop_i w_i$ and $\dd s=\sqrt{g_i}\,\dd\bxi$, we obtain
\[
  \|\Delta_\Gamma u\|_{L^2(\Gamma)}^2
  =\sum_{i=1}^J\int_{D_i}|\Gop_i w_i(\bxi)|^2\,\sqrt{g_i(\bxi)}\,\dd\bxi
  \le C_g\sum_{i=1}^J\|\Gop_i w_i\|_{L^2(D_i)}^2,
\]
where $C_g:=\max_i\|\sqrt{g_i}\|_{L^\infty(D_i)}<\infty$ depends only on the atlas.
Combining with \eqref{eq:w-to-u} yields
\begin{equation}\label{eq:w-by-Gw}
  \|\bm w\|_{\bm H^2(\mathcal D)}
  \le C_{l}C_\Gamma\sqrt{C_g}\,
      \Big(\sum_{i=1}^J\|\Gop_i w_i\|_{L^2(D_i)}^2\Big)^{1/2}.
\end{equation}

\medskip
\noindent\textbf{Step 4: bound $\Gop \bm w$ by $\Gop \bm v$ and the lifting term.}
By linearity, $\Gop_i w_i=\Gop_i v_i-\Gop_i\psi_i$. Hence
\[
  \Big(\sum_{i=1}^J\|\Gop_i w_i\|_{L^2(D_i)}^2\Big)^{1/2}
  \le
  \Big(\sum_{i=1}^J\|\Gop_i v_i\|_{L^2(D_i)}^2\Big)^{1/2}
  +
  \Big(\sum_{i=1}^J\|\Gop_i \psi_i\|_{L^2(D_i)}^2\Big)^{1/2}.
\]
By Lemma~\ref{lem:unif-ellip-atlas}, the coefficients of $\Gop_i$ are bounded on $\overline{D_i}$, hence
$\Gop_i:H^2(D_i)\to L^2(D_i)$ is bounded. Let
\[
  C_0:=\max_{1\le i\le J}\|\Gop_i\|_{\mathcal L(H^2(D_i),L^2(D_i))}<\infty.
\]
Then
\[
  \Big(\sum_{i=1}^J\|\Gop_i \psi_i\|_{L^2(D_i)}^2\Big)^{1/2}
  \le C_0\,\|\bm\psi\|_{\bm H^2(\mathcal D)}
  \le C_0\,C_{\mathrm{lift}}\,\|\Bop\bm v\|_{\mathcal Z},
\]
where we used Lemma~\ref{lem:lifting-atlas} in the last step.

\medskip
\noindent\textbf{Step 5: conclude for $\bm v$.}
Combining \eqref{eq:w-by-Gw} with the estimate from Step~4 yields
\[
  \|\bm w\|_{\bm H^2(\mathcal D)}
  \le A\Big(\sum_{i=1}^J \|\Gop_i v_i\|_{L^2(D_i)}^2\Big)^{1/2}
     + A\,C_0C_{\mathrm{lift}}\|\Bop\bm v\|_{\mathcal Z},
\]
where $A:=C_{l}C_\Gamma\sqrt{C_g}$.
Using $\|\bm v\|_{\bm H^2(\mathcal D)}\le \|\bm w\|_{\bm H^2(\mathcal D)}+\|\bm\psi\|_{\bm H^2(\mathcal D)}$ and
$\|\bm\psi\|_{\bm H^2(\mathcal D)}\le C_{\mathrm{lift}}\|\Bop\bm v\|_{\mathcal Z}$, we further obtain
\[
  \|\bm v\|_{\bm H^2(\mathcal D)}
  \le A\Big(\sum_{i=1}^J \|\Gop_i v_i\|_{L^2(D_i)}^2\Big)^{1/2}
     + B\,\|\Bop\bm v\|_{\mathcal Z},
\]
with $B:=C_{\mathrm{lift}}+A\,C_0C_{\mathrm{lift}}$.
Finally, since
\(
B\|\Bop\bm v\|_{\mathcal Z}
=\frac{B}{\sqrt{\eta}}\big(\sqrt{\eta}\,\|\Bop\bm v\|_{\mathcal Z}\big),
\)
we have
\[
  \|\bm v\|_{\bm H^2(\mathcal D)}
  \le \max\!\left\{A,\frac{B}{\sqrt{\eta}}\right\}
  \left(
     \Big(\sum_{i=1}^J \|\Gop_i v_i\|_{L^2(D_i)}^2\Big)^{1/2}
     + \sqrt{\eta}\,\|\Bop \bm v\|_{\mathcal Z}
  \right).
\]
Therefore the claim holds with $C_{\mathrm{tr}}:=\max\{A,B/\sqrt{\eta}\}$, i.e.,
$C_1=A$ and $C_2=B$ in the statement.
\end{proof}

\begin{lemma}[Upper bound of the training graph norm]\label{lem:upper-graph-atlas}
Assume Lemma~\ref{lem:unif-ellip-atlas} and that each interface map $\Phi_{ji}^e$ is $C^{1}$ and bi-Lipschitz on its neighborhood, i.e.,
$\|D\Phi_{ji}^e\|_{L^\infty(\mathcal N_{i,e})}+\|D(\Phi_{ji}^e)^{-1}\|_{L^\infty(\mathcal N_{j,e})}<\infty$.
Let $\eta>0$. Then there exists $C_U>0$, depending only on $\eta$, the atlas, and the coefficient bounds of $\{\Gop_i\}_{i=1}^J$, such that
for all $\bm v\in \bm H^2(\mathcal D)$,
\[
  \sum_{i=1}^J \|\Gop_i v_i\|_{L^2(D_i)}^2 + \eta \|\Bop \bm v\|_{Z}^2
  \le C_U\,\|\bm v\|_{\bm H^2(\mathcal D)}^2,
\]
where $\Bop_e\bm v=(\delta_e^0(\bm v),\delta_e^n(\bm v))$ and $Z_e=(L^2(\partial_e D))^2$.
\end{lemma}

\begin{proof}
\noindent\textbf{Step 1: interior terms.}
Writing $\Gop_i$ in non-divergence form,
\[
  \Gop_i v = -g_i^{\alpha\beta}\partial_{\alpha\beta}v - b_i^\beta \partial_\beta v,
  \qquad b_i^\beta:=\frac{1}{\sqrt{g_i}}\partial_\alpha(\sqrt{g_i}\,g_i^{\alpha\beta}),
\]
Lemma~\ref{lem:unif-ellip-atlas} yields $g_i^{\alpha\beta},b_i^\beta\in L^\infty(D_i)$.
Hence $\|\Gop_i v_i\|_{L^2(D_i)}\le C_{G,i}\|v_i\|_{H^2(D_i)}$, and summing over $i$ gives
\[
  \sum_{i=1}^J \|\Gop_i v_i\|_{L^2(D_i)}^2 \le C_G\,\|\bm v\|_{\bm H^2(\mathcal D)}^2,
\]
for $C_G:=\max_i C_{G,i}^2$.

\smallskip
\noindent\textbf{Step 2: interface terms.}
Fix an interface $e$ shared by $(i,j)$ and measure $\Bop_e\bm v$ on the chosen side $\partial_e D_{i(e)}$.
By the bi-Lipschitz property of $\Phi_{ji}^e$ and the change-of-variables formula on curves,
\[
  \|\psi\circ\Phi_{ji}^e\|_{L^2(\partial_e D_i)}\le C_{\Phi,e}\|\psi\|_{L^2(\partial_e D_j)}.
\]
Using the trace theorem on Lipschitz domains,
\[
  \|v_\ell\|_{L^2(\partial_e D_\ell)}+\|\partial_{\nu_{\ell,e}}v_\ell\|_{L^2(\partial_e D_\ell)}
  \le C_{\mathrm{Lip}}\|v_\ell\|_{H^2(D_\ell)},\qquad \ell\in\{i,j\},
\]
and the chain rule for $\partial_{\nu_{i,e}}(v_j\circ\Phi_{ji}^e)$, we obtain
\[
  \|\delta_e^0(\bm v)\|_{L^2(\partial_e D_{i(e)})}
  \le C_{e}\big(\|v_i\|_{H^2(D_i)}+\|v_j\|_{H^2(D_j)}\big),
\]
\[
  \|\delta_e^{n}(\bm v)\|_{L^2(\partial_e D_{i(e)})}
  \le C_{e}\big(\|v_i\|_{H^2(D_i)}+\|v_j\|_{H^2(D_j)}\big),
\]
with $C_e$ depending only on the atlas and $\Phi_{ji}^e$.
Therefore,
\[
  \|\Bop_e\bm v\|_{Z_e}^2
  =\|\delta_e^0(\bm v)\|_{L^2(\partial_e D)}^2+\|\delta_e^n(\bm v)\|_{L^2(\partial_e D)}^2
  \le C_{B,e}\big(\|v_i\|_{H^2(D_i)}^2+\|v_j\|_{H^2(D_j)}^2\big).
\]
Summing over $e\in\mathcal E$ and counting incidences yields a factor $d_{\max}$:
\[
  \|\Bop\bm v\|_Z^2 \le C_B\,d_{\max}\,\|\bm v\|_{\bm H^2(\mathcal D)}^2,
\]
where $C_B:=\max_{e\in\mathcal E} C_{B,e}$.

\smallskip
\noindent\textbf{Step 3: combine and insert $\eta$.}
Combining the two estimates gives
\[
  \sum_{i=1}^J \|\Gop_i v_i\|_{L^2(D_i)}^2 + \eta \|\Bop \bm v\|_{Z}^2
  \le \big(C_G+\eta\,C_B d_{\max}\big)\,\|\bm v\|_{\bm H^2(\mathcal D)}^2.
\]
Thus the claim holds with $C_U:=C_G+\eta\,C_B d_{\max}$.
\end{proof}

\begin{assumption}[Interface inverse estimate on the trial space]
\label{ass:inv-atlas}
Consider a sequence of trial spaces $\bm{\mathcal N}^{par}_{\bm M,C_N}$ with total width
$M_{\rm tot}:=\sum_{i=1}^J M_i$.
For each $M_{\rm tot}$, the random parameters are drawn from bounded distributions,
\[
  \|w_{i,m}\|_\infty \le r_{M_{\rm tot}},\qquad |b_{i,m}|\le b_{\max}\quad\text{a.s.},
  \qquad i=1,\dots,J,\ m=1,\dots,M_i,
\]
where the bandwidth $r_{M_{\rm tot}}$ may depend on $M_{\rm tot}$, but remains finite for each fixed $M_{\rm tot}$.

Fix an interface $e\in\mathcal E$ equipped with the arc-length measure.
Let
\[
  S_e:=\{\Bop_e \bm v:\ \bm v\in \bm{\mathcal N}^{par}_{\bm M,C_N}\}\subset \mathcal Z_e
\]
be the finite-dimensional trace-mismatch space induced by the trial class.
We assume that, with high probability over the random draw of features, there exists a constant
$C_{{\rm inv},e}>0$ such that
\[
  \|g\|_{\mathcal Z_e}\le C_{{\rm inv},e}\,\|g\|_{Z_e}\qquad \forall g\in S_e,
\]
and hence, with $C_{\rm inv}:=\max_{e\in\mathcal E}C_{{\rm inv},e}$,
\[
  \|\Bop \bm v\|_{\mathcal Z}\le C_{\rm inv}\,\|\Bop \bm v\|_{Z}\qquad
  \forall \bm v\in \bm{\mathcal N}^{par}_{\bm M,C_N}.
\]
Moreover, along the sequence of trial spaces used in the analysis, $C_{\rm inv}$ is assumed to remain
\emph{uniformly controlled} in the sense that it does not exhibit pathological blow-up as $M_{\rm tot}$ increases.
In particular, $C_{\rm inv}$ may depend on the atlas, the activation regularity, and the effective bandwidth
$r_{M_{\rm tot}}$, as well as on the conditioning of the induced trace bases on each interface.
\end{assumption}

\begin{corollary}[Graph-norm control in terms of the computable $Z$-norm]
\label{cor:graphA-atlas-Z}
Let $\bm u^\ast=(\tilde u_1^\ast,\dots,\tilde u_J^\ast)$ be the pullbacks of the exact mean-zero solution,
so that $\Gop_i \tilde u_i^\ast=\tilde f_i$ and $\Bop\bm u^\ast=0$.
Assume Theorem~\ref{thm:graphA-atlas-trace-eta} and Assumption~\ref{ass:inv-atlas}.
Then for every $\bm v\in \Pi_\diam\bm{\mathcal N}^{par}_{\bm M,C_N}$,
\begin{equation}\label{eq:graphA-atlas-Z}
  \|\bm v-\bm u^\ast\|_{\bm H^2(\mathcal D)}
  \le C_{gn}\Big(
     \Big(\sum_{i=1}^J \|\Gop_i v_i-\tilde f_i\|_{L^2(D_i)}^2\Big)^{1/2}
     + \sqrt{\eta}\,\|\Bop \bm v\|_{Z}\Big),
\end{equation}
where one may take $C_{gn}:=C_{\mathrm{tr}}\max\{1,C_{\mathrm{inv}}\}$.
\end{corollary}

\begin{lemma}[Residual-to-error bound]
\label{lem:err-vs-lossA-atlas}
Assume Corollary~\ref{cor:graphA-atlas-Z}. Then for any $\bm v\in \Pi_\diam\bm{\mathcal N}^{par}_{\bm M,C_N}$,
\[
  \|\bm v-\bm u^\ast\|_{\bm H^2(\mathcal D)}^2
  \le 2 C_{gn}^2\, \mathcal J(\bm v),
\]
where $\mathcal J(\bm v):=
\sum_{i=1}^J \|\Gop_i v_i-\tilde f_i\|_{L^2(D_i)}^2 + \eta\|\Bop \bm v\|_{Z}^2$.
\end{lemma}
\begin{proof}
Apply \eqref{eq:graphA-atlas-Z} and use $(a+b)^2\le 2(a^2+b^2)$.
\end{proof}

\medskip
\noindent\textbf{Step 2: error decomposition.}

\begin{theorem}[Error decomposition]\label{thm:decompA-atlas}
Let $\bm u_{\bm M}$ be defined by \eqref{eq:uM-A-atlas} and let
\[
  \bm u_a \in \arg\min_{\bm v\in \Pi_\diam\bm{\mathcal N}^{par}_{\bm M,C_N}}
  \|\bm v-\bm u^\ast\|_{\bm H^2(\mathcal D)}
\]
be the best broken-$H^2$ approximation in the trial class.
Assume Lemma~\ref{lem:upper-graph-atlas} and Corollary~\ref{cor:graphA-atlas-Z}. Then
\begin{align}
  \|\bm u_{\bm M}-\bm u^\ast\|_{\bm H^2(\mathcal D)}^2
  &\le
  2 C_{gn}^2 C_U\, \|\bm u_a-\bm u^\ast\|_{\bm H^2(\mathcal D)}^2
  + 4 C_{gn}^2 \sup_{\bm v\in \Pi_\diam\bm{\mathcal N}^{par}_{\bm M,C_N}}
  \big|\mathcal J(\bm v)-\widehat{\mathcal J}(\bm v)\big|
  \label{eq:decompA-atlas}\\
  &\quad
  + 2 C_{gn}^2\Big(\widehat{\mathcal J}(\bm u_{\bm M})-\widehat{\mathcal J}(\bm u_a)\Big).
  \nonumber
\end{align}
If the least-squares problem is solved exactly, then
$\widehat{\mathcal J}(\bm u_{\bm M})\le \widehat{\mathcal J}(\bm u_a)$ and the last term is non-positive.
\end{theorem}

\begin{proof}
By Lemma~\ref{lem:err-vs-lossA-atlas},
\[
\|\bm u_{\bm M}-\bm u^\ast\|_{\bm H^2(\mathcal D)}^2
\le 2C_{gn}^2\,\mathcal J(\bm u_{\bm M}).
\]
Decompose $\mathcal J=\widehat{\mathcal J}+(\mathcal J-\widehat{\mathcal J})$ and add/subtract
$\widehat{\mathcal J}(\bm u_a)$ to obtain
\begin{align*}
\mathcal J(\bm u_{\bm M})
&= \widehat{\mathcal J}(\bm u_{\bm M})+(\mathcal J-\widehat{\mathcal J})(\bm u_{\bm M})\\
&\le \widehat{\mathcal J}(\bm u_a)
   +\big(\widehat{\mathcal J}(\bm u_{\bm M})-\widehat{\mathcal J}(\bm u_a)\big)
   + \sup_{\bm v\in\Pi_\diam\bm{\mathcal N}^{par}_{\bm M,C_N}}|\mathcal J(\bm v)-\widehat{\mathcal J}(\bm v)|.
\end{align*}
Moreover,
\[
\widehat{\mathcal J}(\bm u_a)
\le \mathcal J(\bm u_a)+\sup_{\bm v\in\Pi_\diam\bm{\mathcal N}^{par}_{\bm M,C_N}}|\mathcal J(\bm v)-\widehat{\mathcal J}(\bm v)|.
\]
Since $\mathcal J(\bm u^\ast)=0$, we have $\mathcal J(\bm u_a)=\mathcal J(\bm u_a-\bm u^\ast)$, and
Lemma~\ref{lem:upper-graph-atlas} yields
$\mathcal J(\bm u_a)\le C_U\|\bm u_a-\bm u^\ast\|_{\bm H^2(\mathcal D)}^2$.
Combining the above bounds gives \eqref{eq:decompA-atlas}.
\end{proof}

\medskip
\noindent\textbf{Step 3: approximation and statistical bounds.}

\begin{assumption}[Activation and parameter boundedness]\label{ass:rhoA-atlas}
Let $\rho\in C^3(\mathbb R)$. On each patch $D_i\subset\mathbb R^2$, the random feature family is
\[
  \psi_{i,m}(\bxi)=\rho(w_{i,m}\cdot\bxi+b_{i,m}),\qquad m=1,\dots,M_i.
\]
For a given trial space (with total width $M_{\rm tot}$), assume the parameters are almost surely bounded:
\[
  \|w_{i,m}\|_\infty\le r_{M_{\rm tot}},\qquad |b_{i,m}|\le b_{M_{\rm tot}}.
\]
Moreover, $\rho,\rho',\rho'',\rho'''$ are bounded on every bounded interval.
\end{assumption}

\begin{lemma}[Uniform feature bounds on an atlas]\label{lem:featA-atlas}
Assume Lemma~\ref{lem:unif-ellip-atlas} and Assumption~\ref{ass:rhoA-atlas}.
Assume moreover that the interface extensions $\Phi_{ji}^e:\mathcal N_{i,e}\to\mathcal N_{j,e}$
exist as introduced earlier (with $k\ge 2$, hence $\Phi_{ji}^e\in C^1$), and define
\[
  L_\Phi:=\max_{e\in\mathcal E}\max\big\{\|D\Phi_{ji}^e\|_{L^\infty(\mathcal N_{i,e})},\ \|D\Phi_{ij}^e\|_{L^\infty(\mathcal N_{j,e})}\big\}<\infty.
\]
Then there exists a constant $C_{\mathrm{feat}}>0$ depending only on
\[
  r_{M_{\rm tot}},\ b_{M_{\rm tot}},\ \rho,\ \text{the atlas (in particular $\{D_i\}$ and $L_\Phi$)},\
  \text{and coefficient bounds of }\{\Gop_i\}_{i=1}^J,
\]
such that for any $\bm v\in \bm{\mathcal N}^{par}_{\bm M,C_N}$ with patchwise expansions
$v_i=\sum_{m=1}^{M_i}c_{i,m}\psi_{i,m}$ and $\|c_i\|_{\ell^1}\le C_N$,
\[
  \sup_{\bxi\in D_i}|\Gop_i v_i(\bxi)|\le C_N\,C_{\mathrm{feat}},\qquad
  \sup_{e\in\mathcal E}\sup_{\bxi\in \partial_e D_{i(e)}}\|\Bop_e\bm v(\bxi)\|_{\R^2}\le 2C_N\,C_{\mathrm{feat}}.
\]
\end{lemma}

\begin{proof}
\textbf{Step 0: a uniform bound for the feature arguments.}
Since the atlas is finite and each $D_i\subset\mathbb R^2$ is bounded, define
\[
  R:=\max_{1\le i\le J}\ \sup_{\bxi\in D_i} |\bxi|_1 <\infty .
\]
For any feature $\psi_{i,m}(\bxi)=\rho(w_{i,m}\cdot\bxi+b_{i,m})$ and any $\bxi\in D_i$,
\[
  |w_{i,m}\cdot\bxi+b_{i,m}|
  \le \|w_{i,m}\|_\infty\,|\bxi|_1 + |b_{i,m}|
  \le r_{M_{\rm tot}}\,R + b_{M_{\rm tot}}=:T .
\]
Hence all arguments of $\rho$ remain in the fixed compact interval $[-T,T]$.

\smallskip
\textbf{Step 1: uniform $L^\infty$ bounds for $\psi_{i,m}$ and its derivatives on each patch.}
Let
\[
  K_0:=\sup_{|t|\le T}|\rho(t)|,\qquad
  K_1:=\sup_{|t|\le T}|\rho'(t)|,\qquad
  K_2:=\sup_{|t|\le T}|\rho''(t)|,
\]
which are finite by Assumption~\ref{ass:rhoA-atlas}.
With $t(\bxi):=w_{i,m}\cdot\bxi+b_{i,m}$, the chain rule gives
\[
  \partial_{\xi_\alpha}\psi_{i,m}(\bxi)=\rho'(t(\bxi))\,w_{i,m,\alpha},
  \qquad
  \partial_{\xi_\alpha\xi_\beta}\psi_{i,m}(\bxi)=\rho''(t(\bxi))\,w_{i,m,\alpha}w_{i,m,\beta}.
\]
Using $|w_{i,m,\alpha}|\le \|w_{i,m}\|_\infty\le r_{M_{\rm tot}}$, we obtain
\[
  \sup_{D_i}|\psi_{i,m}|\le K_0,\qquad
  \sup_{D_i}|\partial_{\xi_\alpha}\psi_{i,m}|\le K_1\,r_{M_{\rm tot}},\qquad
  \sup_{D_i}|\partial_{\xi_\alpha\xi_\beta}\psi_{i,m}|\le K_2\,r_{M_{\rm tot}}^2 .
\]

\smallskip
\textbf{Step 2: uniform $L^\infty$ bound for $\Gop_i\psi_{i,m}$.}
Expanding $\Gop_i$ in non-divergence form,
\[
  \Gop_i v = -g_i^{\alpha\beta}\,\partial_{\xi_\alpha\xi_\beta}v - b_i^\beta\,\partial_{\xi_\beta}v,
  \qquad
  b_i^\beta:=\frac{1}{\sqrt{g_i}}\partial_{\xi_\alpha}\big(\sqrt{g_i}\,g_i^{\alpha\beta}\big),
\]
Lemma~\ref{lem:unif-ellip-atlas} yields bounded coefficients on $\overline{D_i}$.
Let
\[
  A:=\max_{i}\max_{\alpha,\beta}\|g_i^{\alpha\beta}\|_{L^\infty(D_i)},\qquad
  B:=\max_{i}\max_{\beta}\|b_i^\beta\|_{L^\infty(D_i)}.
\]
Then Step~1 gives
\[
  \sup_{D_i}|\Gop_i\psi_{i,m}|
  \le A\sum_{\alpha,\beta}\sup_{D_i}|\partial_{\xi_\alpha\xi_\beta}\psi_{i,m}|
      +B\sum_{\beta}\sup_{D_i}|\partial_{\xi_\beta}\psi_{i,m}|
  \le C^{(1)}\big(r_{M_{\rm tot}}^2+r_{M_{\rm tot}}\big),
\]
where $C^{(1)}$ depends only on $A,B,K_1,K_2$ (hence on the atlas and $\rho$).

\smallskip
\textbf{Step 3: extend the bound to a trial function $v_i=\sum_m c_{i,m}\psi_{i,m}$.}
By linearity,
\[
  \Gop_i v_i=\sum_{m=1}^{M_i}c_{i,m}\,\Gop_i\psi_{i,m},
\]
hence
\[
  \sup_{D_i}|\Gop_i v_i|
  \le \|c_i\|_{\ell^1}\ \sup_{m}\sup_{D_i}|\Gop_i\psi_{i,m}|
  \le C_N\,C^{(2)},
\]
where $C^{(2)}:=\sup_{i,m}\sup_{D_i}|\Gop_i\psi_{i,m}|$ is controlled by Step~2.

\smallskip
\textbf{Step 4: uniform bound for the interface mismatch $\Bop_e\bm v=(\delta_e^0,\delta_e^n)$.}
Fix an interface $e=\Gamma_i\cap\Gamma_j$ and its chosen side $i(e)=i$ so that $\Bop_e\bm v$ is evaluated on $\partial_e D_i$.
For the value mismatch,
\[
  \sup_{\partial_e D_i}|v_i|
  \le \sum_m |c_{i,m}|\sup_{D_i}|\psi_{i,m}|
  \le C_N K_0,
  \qquad
  \sup_{\partial_e D_i}|v_j\circ \Phi_{ji}^e|
  \le \sup_{D_j}|v_j|\le C_N K_0,
\]
so
\[
  \sup_{\partial_e D_i}|\delta_e^0(\bm v)|\le 2C_NK_0.
\]

For the normal-derivative mismatch, let $\nu_{i,e}(\bxi)$ be the outward unit normal to $\partial_e D_i$.
Using $\partial_{\nu_{i,e}}v_i=\nu_{i,e}\cdot\nabla v_i$ and the trace bounds from Step~1,
\[
  \sup_{\partial_e D_i}|\partial_{\nu_{i,e}} v_i|
  \le \sup_{\partial_e D_i}|\nabla v_i|
  \le \sum_m |c_{i,m}|\sup_{D_i}|\nabla \psi_{i,m}|
  \le C_N\,(2K_1 r_{M_{\rm tot}}).
\]
Moreover, by the chain rule,
\[
  \partial_{\nu_{i,e}}\big(v_j\circ \Phi_{ji}^e\big)(\bxi)
  = \nabla v_j(\Phi_{ji}^e(\bxi))\cdot D\Phi_{ji}^e(\bxi)\,\nu_{i,e}(\bxi),
\]
hence
\[
  \sup_{\partial_e D_i}\Big|\partial_{\nu_{i,e}}\big(v_j\circ \Phi_{ji}^e\big)\Big|
  \le \sup_{D_j}|\nabla v_j|\ \|D\Phi_{ji}^e\|_{L^\infty(\mathcal N_{i,e})}
  \le C_N\,(2K_1 r_{M_{\rm tot}})\,L_\Phi .
\]
Therefore
\[
  \sup_{\partial_e D_i}|\delta_e^n(\bm v)|
  \le C_N\,(2K_1 r_{M_{\rm tot}})\,(1+L_\Phi).
\]

Combining the bounds for $(\delta_e^0,\delta_e^n)$ and enlarging constants if necessary yields
\[
  \sup_{\bxi\in \partial_e D_i}\|\Bop_e\bm v(\bxi)\|_{\R^2}
  \le 2C_N\,C_{\mathrm{feat}},
\]
with $C_{\mathrm{feat}}$ depending only on $r_{M_{\rm tot}},b_{M_{\rm tot}},\rho$, the atlas (through $R$ and $L_\Phi$),
and coefficient bounds of $\{\Gop_i\}$ (through $A,B$).
\end{proof}

\begin{theorem}[Patchwise Sobolev approximation on an atlas]\label{thm:approxA-atlas}
Assume Assumption~\ref{ass:rhoA-atlas}. 
Assume further that the exact mean-zero solution satisfies
$u^\ast\in H^{\frac52+\varepsilon}(\Gamma)$ for some $\varepsilon>0$.
Then for any $\epsilon_A>0$ and $\delta\in(0,1)$ there exist a coefficient radius $C_N>0$
and widths $\bm M=(M_1,\dots,M_J)$ such that, with probability at least $1-\delta$
over the random draw of all patchwise features,
\[
  \inf_{\bm v\in \Pi_\diam\bm{\mathcal N}^{par}_{\bm M,C_N}}
  \|\bm v-\bm u^\ast\|_{\bm H^2(\mathcal D)}
  \le \epsilon_A .
\]
\end{theorem}

\begin{proof}
\noindent\textbf{Step 1: pull back the target to the parameter charts.}
Let $\tilde u_i^\ast:=u^\ast\circ X_i$ on $D_i$.
Since $u^\ast\in H^{\frac52+\varepsilon}(\Gamma)$ and the atlas maps $X_i$ are $C^k$ diffeomorphisms
on each patch (with $k$ large enough so that composition is stable at order $s=\frac52+\varepsilon$),
the pullback is bounded on each chart:
\[
  \tilde u_i^\ast \in H^{\frac52+\varepsilon}(D_i),
  \qquad
  \|\tilde u_i^\ast\|_{H^{\frac52+\varepsilon}(D_i)}\le C_{X,i}\,\|u^\ast\|_{H^{\frac52+\varepsilon}(\Gamma_i)}.
\]
Here $C_{X,i}$ depends only on the atlas (chart regularity and Jacobian bounds).

\medskip
\noindent\textbf{Step 2: randomized Sobolev approximation on each patch.}
Fix $\epsilon_A>0$ and $\delta\in(0,1)$.
Set target accuracies and failure probabilities
\[
  \epsilon_i := \frac{\epsilon_A}{C_\Pi\sqrt{J}},
  \qquad
  \delta_i := \frac{\delta}{J},
\]
where $C_\Pi\ge 1$ is the stability constant of the mean-zero projection $\Pi_\diam$
(see Step~4 below).
Apply the patchwise randomized Sobolev approximation theorem
(e.g., \cite[Thm.~4.12]{AGRANN2025} with $(d,s)=(2,2)$) on the bounded domain $D_i$ to the function
$\tilde u_i^\ast\in H^{\frac52+\varepsilon}(D_i)$.
Under Assumption~\ref{ass:rhoA-atlas} (activation regularity on bounded intervals and bounded parameters),
this theorem yields: there exist a width $M_i$ and a coefficient radius $C_{N,i}$ such that, with probability
at least $1-\delta_i$ over the random draw of the $M_i$ features on patch $i$, one can find
\[
  v_i \in \mathcal N^{par}_{M_i,C_{N,i}}(D_i)
  \quad\text{with}\quad
  \|v_i-\tilde u_i^\ast\|_{H^2(D_i)} \le \epsilon_i .
\]
Define $C_N:=\max_{1\le i\le J} C_{N,i}$ and enlarge radii if needed, so that all constructed $v_i$ belong to
$\mathcal N^{par}_{M_i,C_N}(D_i)$ simultaneously.

\medskip
\noindent\textbf{Step 3: combine patches by a union bound.}
Let $\mathcal E_i$ be the event that the above patchwise approximation succeeds on patch $i$.
Then $\mathbb P(\mathcal E_i)\ge 1-\delta_i$ and hence
\[
  \mathbb P\Big(\bigcap_{i=1}^J \mathcal E_i\Big)
  \ge 1-\sum_{i=1}^J \delta_i
  =1-\delta .
\]
On the intersection event, define the broken function $\bm v:=(v_1,\dots,v_J)\in \bm{\mathcal N}^{par}_{\bm M,C_N}$.
By definition of the broken norm,
\[
  \|\bm v-\bm u^\ast\|_{\bm H^2(\mathcal D)}^2
  =\sum_{i=1}^J \|v_i-\tilde u_i^\ast\|_{H^2(D_i)}^2
  \le \sum_{i=1}^J \epsilon_i^2
  = \frac{\epsilon_A^2}{C_\Pi^2}.
\]

\medskip
\noindent\textbf{Step 4: enforce mean-zero via $\Pi_\diam$ without losing accuracy.}
Recall that $\Pi_\diam$ subtracts a global constant to enforce $\int_\Gamma(\Pi_\diam \bm v)\,ds=0$:
for $\bm v\in \bm H^2(\mathcal D)$ set
\[
  (\Pi_\diam \bm v)_i := v_i - c(\bm v),
  \qquad
  c(\bm v):=|\Gamma|^{-1}\int_\Gamma v\,ds,
\]
where $v$ denotes the patchwise surface function induced by $\bm v$ (well-defined a.e.\ on each patch).
Since $u^\ast$ is mean-zero, $c(\bm u^\ast)=0$ and thus
\[
  \Pi_\diam \bm v - \bm u^\ast = (\bm v-\bm u^\ast) - c(\bm v)\,(1,\dots,1).
\]
By Cauchy--Schwarz,
\[
  |c(\bm v)|
  = |\Gamma|^{-1}\Big|\int_\Gamma (v-u^\ast)\,ds\Big|
  \le |\Gamma|^{-1/2}\,\|v-u^\ast\|_{L^2(\Gamma)}
  \le C\,\|\bm v-\bm u^\ast\|_{\bm H^2(\mathcal D)},
\]
and constants have finite $\bm H^2(\mathcal D)$ norm. Therefore $\Pi_\diam$ is bounded on $\bm H^2(\mathcal D)$:
\[
  \|\Pi_\diam \bm v-\bm u^\ast\|_{\bm H^2(\mathcal D)}
  \le C_\Pi\,\|\bm v-\bm u^\ast\|_{\bm H^2(\mathcal D)}
\]
for an atlas-dependent constant $C_\Pi\ge 1$.
Combining with Step~3 yields
\[
  \|\Pi_\diam \bm v-\bm u^\ast\|_{\bm H^2(\mathcal D)}\le \epsilon_A
\]
on the same event of probability at least $1-\delta$.
This proves the claim.
\end{proof}

\begin{theorem}[Uniform statistical bound on an atlas]\label{thm:statA-atlas}
Assume Lemma~\ref{lem:featA-atlas} and $\tilde f_i\in L^\infty(D_i)$ for all $i$.
Let $\eta>0$ and $\delta\in(0,1)$. Set $M_{\rm tot}:=\sum_{i=1}^J M_i$.
For each patch $i$, draw $\{\bxi_{i,n}\}_{n=1}^{N_i}\overset{\rm i.i.d.}{\sim}U(D_i)$.
For each interface $e\in\mathcal E$, draw $\{\bxi_{e,n}\}_{n=1}^{N_e}\overset{\rm i.i.d.}{\sim}U(\partial_e D)$ (with respect to $\dd l$).
Then there exist absolute constants $c_0,C>0$ such that, with probability at least $1-\delta$
(over the random collocation points),
\begin{align*}
  &\sup_{\bm v\in \Pi_\diam\bm{\mathcal N}^{par}_{\bm M,C_N}}
  \big|\mathcal J(\bm v)-\widehat{\mathcal J}(\bm v)\big|
  \le \\
 & C\, C_N^2 \Big(C_{\mathrm{feat}}+\max_i\|\tilde f_i\|_{L^\infty(D_i)}\Big)^2
  \left(
     \sum_{i=1}^J M_i\,\sqrt{\frac{\log(c_0 M_{\mathrm{tot}}/\delta)}{N_i}}
     + \eta\sum_{e\in\mathcal E} M_e\,\sqrt{\frac{\log(c_0 M_{\mathrm{tot}}/\delta)}{N_e}}
  \right),
\end{align*}
where $\mathcal J(\bm v)=\sum_{i=1}^J \|\Gop_i v_i-\tilde f_i\|_{L^2(D_i)}^2 + \eta \|\Bop \bm v\|_{Z}^2$,
$\widehat{\mathcal J}$ is its empirical counterpart, and $M_e:=M_{i(e)}+M_{j(e)}$.
\end{theorem}

\begin{proof}
\noindent\textbf{Step 1: patch residual terms.}
Fix a patch $i$. For $\bm v$ in the trial class, $v_i$ has the form
\[
v_i(\bxi)=\sum_{m=1}^{M_i} c_{i,m}\psi_{i,m}(\bxi),\qquad \|c_i\|_{\ell^1}\le C_N,
\]
and the pointwise residual is
\[
r_i(\bxi):=\Gop_i v_i(\bxi)-\tilde f_i(\bxi)
=\sum_{m=1}^{M_i} c_{i,m}\,\Gop_i\psi_{i,m}(\bxi)-\tilde f_i(\bxi).
\]
By Lemma~\ref{lem:featA-atlas} and $\tilde f_i\in L^\infty(D_i)$,
\[
|r_i(\bxi)|\le C_N C_{\rm feat}+\|\tilde f_i\|_{L^\infty(D_i)} \quad \forall \bxi\in D_i,
\]
hence $|r_i(\bxi)|^2$ is uniformly bounded on the class.

Let $\mu_i$ be the normalized uniform probability measure on $D_i$.
Applying a uniform concentration bound for squared least-squares losses over an $M_i$-dimensional linear model with bounded features
(e.g., \cite[Thm.~4.14]{AGRANN2025}) yields: there exist absolute constants $c_0,C>0$ such that, with probability at least
$1-\delta_i$,
\[
\sup_{\bm v\in \Pi_\diam\bm{\mathcal N}^{par}_{\bm M,C_N}}
\left|
\int_{D_i}|r_i|^2\,d\mu_i-\frac1{N_i}\sum_{n=1}^{N_i}|r_i(\bxi_{i,n})|^2
\right|
\le
C\,C_N^2\Big(C_{\rm feat}+\|\tilde f_i\|_{L^\infty(D_i)}\Big)^2
\,M_i\sqrt{\frac{\log(c_0 M_{\rm tot}/\delta_i)}{N_i}}.
\]
Multiplying both sides by $|D_i|$ converts this bound to the scaled quantities
$\|r_i\|_{L^2(D_i)}^2$ and $\frac{|D_i|}{N_i}\sum|r_i(\bxi_{i,n})|^2$ that appear in $\mathcal J$ and $\widehat{\mathcal J}$.

\medskip
\noindent\textbf{Step 2: interface mismatch terms.}
Fix an interface $e$ with orientation $i(e)$ and edge $\partial_e D$ equipped with arc-length measure $dl$,
and let $\mu_e$ be the normalized uniform probability measure on $\partial_e D$.
For $\bm v$ in the trial class, the interface mismatch is
\[
\Bop_e\bm v=(\delta_e^0(\bm v),\delta_e^n(\bm v)):\partial_e D\to\mathbb R^2,
\]
and depends only on coefficients from the two incident patches; denote $M_e:=M_{i(e)}+M_{j(e)}$.
By Lemma~\ref{lem:featA-atlas},
\[
\sup_{\bxi\in\partial_e D}\|\Bop_e\bm v(\bxi)\|_{\mathbb R^2}\le 2C_N C_{\rm feat},
\]
so the squared loss $\ell_e(\bxi):=\|\Bop_e\bm v(\bxi)\|_{\mathbb R^2}^2$ is uniformly bounded.

Applying the same uniform concentration theorem on the 1D domain $(\partial_e D,\mu_e)$
(either in its vector-valued form or by applying the scalar form to each component and summing)
yields: with probability at least $1-\delta_e$,
\[
\sup_{\bm v\in \Pi_\diam\bm{\mathcal N}^{par}_{\bm M,C_N}}
\left|
\int_{\partial_e D}\|\Bop_e\bm v\|_{\mathbb R^2}^2\,d\mu_e
-\frac1{N_e}\sum_{n=1}^{N_e}\|\Bop_e\bm v(\bxi_{e,n})\|_{\mathbb R^2}^2
\right|
\le
C\,C_N^2\Big(C_{\rm feat}+\max_i\|\tilde f_i\|_{L^\infty(D_i)}\Big)^2
\,M_e\sqrt{\frac{\log(c_0 M_{\rm tot}/\delta_e)}{N_e}}.
\]
Multiplying by $|\partial_e D|$ converts this bound to the scaled interface terms used in
$\mathcal J$ and $\widehat{\mathcal J}$, and multiplying further by $\eta$ yields the weighted interface contribution.

\medskip
\noindent\textbf{Step 3: union bound and summation.}
Choose $\delta_i=\delta/(2J)$ for $i=1,\dots,J$ and $\delta_e=\delta/(2|\mathcal E|)$ for $e\in\mathcal E$.
A union bound implies that all patch and interface deviation bounds hold simultaneously with probability at least $1-\delta$.
Summing the bounds over $i$ and $e$ and absorbing harmless factors such as $|D_i|$, $|\partial_e D|$ and the constants from the
vector-valued reduction into $C,c_0$ yields the stated inequality.
\end{proof}

\begin{theorem}[Convergence]\label{thm:convA-atlas}
Assume the hypotheses of Theorem~\ref{thm:approxA-atlas} and Theorem~\ref{thm:statA-atlas}, as well as
Corollary~\ref{cor:graphA-atlas-Z} and Lemma~\ref{lem:upper-graph-atlas}.
Fix $\epsilon>0$ and $\delta\in(0,1)$.
Choose widths $\bm M$ (and $C_N$) so that Theorem~\ref{thm:approxA-atlas} holds with
\[
  \epsilon_A=\frac{\epsilon}{2 C_{gn}\sqrt{C_U}}
\]
and probability at least $1-\delta/2$ over the random features.
Choose sample sizes $(N_i)_{i=1}^J$ and $(N_e)_{e\in\mathcal E}$ so that the right-hand side of
Theorem~\ref{thm:statA-atlas} is bounded by
\[
  \frac{\epsilon^2}{8C_{gn}^2}
\]
with probability at least $1-\delta/2$ over the collocation points.
If the empirical least-squares problem \eqref{eq:uM-A-atlas} is solved exactly, then
\[
  \|\bm u_{\bm M}-\bm u^\ast\|_{\bm H^2(\mathcal D)}\le \epsilon
\]
with probability at least $1-\delta$ (over both the random features and the collocation points).
\end{theorem}

\begin{proof}
By the error decomposition \eqref{eq:decompA-atlas} and exact empirical least squares,
the optimization term is non-positive, hence
\[
  \|\bm u_{\bm M}-\bm u^\ast\|_{\bm H^2(\mathcal D)}^2
  \le
  2 C_{gn}^2 C_U\, \|\bm u_a-\bm u^\ast\|_{\bm H^2(\mathcal D)}^2
  + 4 C_{gn}^2 \sup_{\bm v\in \Pi_\diam\bm{\mathcal N}^{par}_{\bm M,C_N}}
  \big|\mathcal J(\bm v)-\widehat{\mathcal J}(\bm v)\big|.
\]
On the event that Theorem~\ref{thm:approxA-atlas} holds with $\epsilon_A=\epsilon/(2C_{gn}\sqrt{C_U})$,
we have $\|\bm u_a-\bm u^\ast\|_{\bm H^2(\mathcal D)}\le \epsilon_A$, so the approximation term is bounded by
\[
  2 C_{gn}^2 C_U\,\epsilon_A^2=\frac{\epsilon^2}{2}.
\]
On the event that Theorem~\ref{thm:statA-atlas} holds with bound $\epsilon^2/(8C_{gn}^2)$,
the statistical term is bounded by
\[
  4 C_{gn}^2\cdot \frac{\epsilon^2}{8C_{gn}^2}=\frac{\epsilon^2}{2}.
\]
Therefore, on the intersection of these two events,
\(
\|\bm u_{\bm M}-\bm u^\ast\|_{\bm H^2(\mathcal D)}^2\le \epsilon^2
\),
hence
\(
\|\bm u_{\bm M}-\bm u^\ast\|_{\bm H^2(\mathcal D)}\le \epsilon.
\)
Finally, a union bound yields that the intersection holds with probability at least
$(1-\delta/2)+(1-\delta/2)-1=1-\delta$.
\end{proof}

\begin{remark}[Extension beyond the Laplace--Beltrami equation]\label{rem:general-elliptic}
The analysis above extends to linear second-order \emph{strongly elliptic} surface operators of the form
\[
  \mathcal L u := -\nabla_\Gamma\cdot(A\nabla_\Gamma u) + b\cdot \nabla_\Gamma u + c\,u,
\]
where $A$ is a tangential, symmetric, uniformly positive definite tensor field on $\Gamma$, $b$ is a tangential vector field,
and $c\in L^\infty(\Gamma)$.
Assume the coefficients are sufficiently smooth (e.g.\ $C^{k-1}$ so that the pulled-back chart operators have bounded coefficients),
and that an $H^2$ graph estimate holds on the appropriate quotient space, namely: there exists $C>0$ such that for all $u\in H^2(\Gamma)$
orthogonal to $\ker(\mathcal L)$ (e.g.\ mean-zero if $\ker(\mathcal L)=\mathrm{span}\{1\}$),
\[
  \|u\|_{H^2(\Gamma)} \le C\,\|\mathcal L u\|_{L^2(\Gamma)}.
\]
Under these assumptions, all steps (chartwise uniform ellipticity, lifting/gluing, graph-norm control, and the subsequent
approximation/statistical error decomposition) carry over by replacing $\Delta_\Gamma$ with $\mathcal L$ and $\{\Gop_i\}$ with
the corresponding pulled-back local operators, with constants depending additionally on coefficient bounds of $A,b,c$.
\end{remark}

\subsubsection{Level-set-based approaches and point-cloud-based RaNN approaches.}\label{subsec:ls-pc}

When local parametrizations of $\Gamma$ are unavailable or impractical to construct, we instead work with an implicit representation. We begin by specifying the level-set description of $\Gamma$ and then present
the associated expressions for computing surface differential operators. Let $\Omega\subset\mathbb R^3$ be a bounded domain and let $\phi\in C^2(\overline{\Omega})$ be a level-set function
such that $\nabla\phi(\bX)\neq 0$ in a neighborhood of the zero level. We consider the surface
\begin{align}
\Gamma=\{\bX=(x,y,z)^\top\in\overline{\Omega}\,:\,\phi(\bX)=0\}.
\end{align}
Define the unit normal and tangential projector
\begin{align}
\bn(\bX)=\frac{\nabla\phi(\bX)}{\|\nabla\phi(\bX)\|_2},\qquad
\bP(\bX)=\bI-\bn(\bX)\bn(\bX)^\top .
\end{align}
Different sign conventions exist for mean curvature in the literature.
All formulas here are consistent with $H=\frac12\nabla\cdot\bn$, so that $\nabla\cdot\bn=2H$.

For a surface function $u:\Gamma\to\mathbb R$, let $\tilde u$ be any $C^2$ extension to a neighborhood of $\Gamma$
satisfying $\tilde u|_\Gamma=u$. Then (see, e.g., \cite{Goldman2005Curvature, Dziuk2013fem})
\begin{align}
\nabla_\Gamma u(\bX) &= \bP(\bX)\nabla \tilde u(\bX), \label{eq:ls-grad}\\
\Delta_\Gamma u(\bX) &= \nabla_\Gamma\cdot(\bP(\bX)\nabla \tilde u(\bX)) \nonumber\\
&= \Delta \tilde u(\bX) - 2H(\bX)\,\nabla \tilde u(\bX)\cdot \bn(\bX)
   - \bn(\bX)^\top(\nabla^2 \tilde u(\bX))\bn(\bX), \label{eq:ls-lb}
\end{align}
where $\nabla$ and $\Delta$ denote the standard gradient and Laplacian in $\mathbb R^3$, $\nabla_\Gamma\cdot$ denotes the surface divergence and $\nabla_\Gamma u=\bP\nabla \tilde u$..

\noindent\textbf{Level-set-based approach.}
When the available geometric description is the level-set function $\phi$, we solve the surface PDE on
$\Gamma=\{\phi=0\}$ by an embedding-type RaNN defined in $\mathbb R^3$. We construct a one-hidden-layer randomized neural network with three inputs and
one output:
\begin{align}\label{eq:rann-ambient}
u_{imp}(\bX)=\bc\,\rho(\bw\bX+\bb),\qquad \bX\in\Omega,
\end{align}
where $\bw\in\mathbb{R}^{M\times 3}$ has i.i.d.\ entries drawn from $U(-r_x,r_x)$, $\bc\in\mathbb R^{1\times M}$ is trainable, and
\begin{align}
\bb:=-(\bw\odot \bB^{\top})\,\mathbf 1_3,\qquad \mathbf 1_3=(1,1,1)^\top,
\end{align}
with $\bB=[\bB_1,\dots,\bB_M]\in\mathbb R^{3\times M}$ whose columns are drawn i.i.d.\ from the uniform distribution on an
axis-aligned bounding box $\widetilde D\subset\mathbb R^3$ containing $\Gamma$ (i.e., $\bB_m\sim U(\widetilde D)$).

In this setting we take the extension $\tilde u=u_{imp}$ in \eqref{eq:ls-grad}--\eqref{eq:ls-lb}, so that the required
ambient derivatives $\nabla \tilde u$ and $\nabla^2\tilde u$ are computed directly from the network.
Let $P_t=\{\bX_n\}_{n=1}^{N}\subset\Gamma$ be collocation points on the surface.
Enforcing the Laplace--Beltrami equation $-\Delta_\Gamma u=f$ at $\{\bX_n\}$ yields
\begin{align}
-\Big(\Delta  u_{imp}(\bX_n)  &- 2H(\bX_n)\,\nabla  u_{imp}(\bX_n)\cdot \bn(\bX_n)
- \bn(\bX_n)^\top(\nabla^2  u_{imp}(\bX_n))\bn(\bX_n)\Big)
= f(\bX_n), \quad \bX_n\in P_t. \label{eq:ls-colloc}
\end{align}
Since \eqref{eq:rann-ambient} is linear in $\bc$, the collocation system \eqref{eq:ls-colloc} is solved in the least-squares sense
for the output weights.

\medskip
Methods \eqref{m_ti_par_bc} (parametrization-based) and \eqref{eq:ls-colloc} (level-set-based) have complementary advantages.
When $\Gamma$ is highly complex, possibly with self-intersections, or when a high-quality global parametrization is unavailable,
the level-set formulation is typically more flexible because it avoids constructing an explicit atlas.
For smooth surfaces equipped with stable parametrizations, the atlas formulation is often more accurate and efficient, since sampling
in parameter space is exact and does not require auxiliary surface-extraction procedures.
In practice, sampling collocation points on a level-set surface may rely on additional algorithms such as DistMesh-type
distance-function meshing \cite{Persson2004distmesh} or Marching Cubes \cite{Lorensen1998mc}. These steps introduce extra computational
cost and may reduce sampling accuracy, which can in turn affect the final PDE solution (see Example~\ref{ex_ti_phi}).
We emphasize that the present paper focuses on presenting the numerical formulations; the choice between the two approaches should be
guided by the specific geometry representation available in the target application.

\smallskip
\noindent\textbf{Point-cloud-based RaNN approach.}
Finally, we consider the setting where the surface is given only through a finite set of point locations
$P=\{\bX_n\}_{n=1}^{N}\subset\Gamma$.
To evaluate the operator \eqref{eq:ls-lb} we reconstruct local geometric quantities from the point cloud.
Specifically, we employ a local quadratic interpolation \cite{Cazals2005fitting}
to obtain approximations of the unit normals $\bn(\bX_n)$ and mean curvature $H(\bX_n)$ (e.g., via the first and second
fundamental forms; see \cite{Walker2015fs}).
We then use the same ambient RaNN ansatz \eqref{eq:rann-ambient} and enforce \eqref{eq:ls-colloc} at the available points
$\bX_n\in P$ to determine $\bc$ in the least-squares sense.

\subsection{Time-dependent partial differential equations}

In this section, we extend the proposed approach to time-dependent surface PDEs using the heat equation as an illustrative example.

Let $\Gamma$ be a static surface and $I=[0,T]$. We consider
\begin{subequations}\label{eq:heat-surface}
\begin{align}
\partial_t u - \Delta_\Gamma u &= f \quad \text{on } I\times \Gamma,\\
u(0,\cdot) &= u_0 \quad \text{on } \Gamma,
\end{align}
\end{subequations}
where $f=f(t,\bX)$ is a given source term and $u_0=u_0(\bX)$ is the initial condition.

\medskip
\noindent\textbf{Parametrization-based approach.}
We restrict attention to \emph{closed} surfaces $\Gamma\subset\mathbb R^3$ described by a finite atlas of $C^k$
parametrizations
\[
  X_i:D_i\to \Gamma_i\subset\Gamma,\qquad i=1,\dots,J,
\]
where $D_i\subset\mathbb R^2$ are bounded Lipschitz domains and the patches form a non-overlapping decomposition
$\Gamma=\bigcup_{i=1}^J\Gamma_i$ as in \eqref{eq:patch-decomp}. We use the chart operators $\Gop_i$ defined in
\eqref{eq:opGi}, so that $(-\Delta_\Gamma u)\circ X_i=\Gop_i(u\circ X_i)$.

We directly extend \eqref{no_boundary}  to solve the heat equation, and define a \emph{single} space--time randomized neural network on embedded coordinates:
\[
  \hat u(t,\bX):=\bc\,\rho(\bw_t\,t+\bw_x\,\bX+\bb),\qquad (t,\bX)\in I\times \mathbb R^3,
\]
where $\bw_t\in\mathbb R^{M\times 1}$ and $\bw_x\in\mathbb R^{M\times 3}$ are drawn independently from
$U(-r_t,r_t)$ and $U(-r_x,r_x)$, respectively, $\bc\in\mathbb R^{1\times M}$ is trainable, and
\[
  \bb:=-(\bw\odot \bB^{\top})\,\mathbf 1_4,\qquad \bw=[\bw_t,\bw_x]\in\mathbb R^{M\times 4},\qquad \mathbf 1_4=(1,1,1,1)^\top,
\]
with $\bB\in\mathbb R^{4\times M}$ chosen to cover the range of $I\times \widetilde D$, where $\widetilde D$ is a bounding box
of $\Gamma$.

On each patch we work with the pullback
\[
  u_i(t,\bxi):=\hat u\big(t,X_i(\bxi)\big),\qquad (t,\bxi)\in I\times D_i.
\]
Since $\hat u$ is globally defined on $\Gamma$, the traces from adjacent patches agree automatically for every $t\in I$.
Hence no separate interface mismatch enforcement is required.

We solve the atlas-form heat equation by collocation: find $\hat u$ such that, for each $i=1,\dots,J$,
\begin{subequations}\label{eq:heat-atlas-nobdry}
\begin{align}
\partial_t u_i(t,\bxi)+\Gop_i u_i(t,\bxi) &= (f\circ X_i)(t,\bxi)\quad &&\text{in } I\times D_i,\\
u_i(0,\bxi) &= (u_0\circ X_i)(\bxi)\quad &&\text{in } D_i.
\end{align}
\end{subequations}
Enforcing \eqref{eq:heat-atlas-nobdry} at interior space--time collocation points yields a linear least-squares system
for the shared output weights $\bc$.

\medskip
\noindent\textbf{Level-set-based approach.}
Assume $\Gamma=\{\bX\in\Omega:\phi(\bX)=0\}$ with $\phi\in C^2(\overline{\Omega})$ and $\nabla\phi\neq 0$ near $\Gamma$.
We extend the embedding-type RaNN by augmenting the input with time:
\begin{align}\label{eq:rann-st-imp}
u_{\mathrm{imp}}(t,\bX)=\bc\,\rho\!\left(\bw_t\,t+\bw_x\,\bX+\bb\right),
\qquad (t,\bX)\in I\times\Omega,
\end{align}
where $\bw_t\in\mathbb R^{M\times 1}$ and $\bw_x\in\mathbb R^{M\times 3}$ are drawn independently from
$U(-r_t,r_t)$ and $U(-r_x,r_x)$, respectively, $\bc\in\mathbb R^{1\times M}$ is trainable, and
\[
  \bb:=-(\bw\odot\bB^\top)\,\mathbf 1_4,\qquad
  \bw=[\bw_t,\bw_x]\in\mathbb R^{M\times 4},\qquad \mathbf 1_4=(1,1,1,1)^\top,
\]
with $\bB=[\bB_1,\dots,\bB_M]\in\mathbb R^{4\times M}$ chosen to cover $I\times\widetilde D$, where $\widetilde D$ is an
axis-aligned bounding box containing $\Gamma$ (e.g., $\bB_m\sim U(I\times\widetilde D)$).

For each fixed $t\in I$ we take the extension $\tilde u(\cdot)=u_{\mathrm{imp}}(t,\cdot)$ in
\eqref{eq:ls-grad}--\eqref{eq:ls-lb}. Since $\Gamma$ is static, the geometric quantities $\bn(\bX)$ and $H(\bX)$ are
computed from $\phi$ and do \emph{not} depend on $t$.

Let $\mathring P_t=P_T\times P_\Gamma\subset I\times\Gamma$ be the space--time collocation set and
$\mathring P_t^0=\{0\}\times P_\Gamma$ the initial set. Enforcing the heat equation at $(t_i,\bX_i)\in\mathring P_t$ yields
\begin{subequations}\label{eq:heat-colloc-imp}
\begin{align}
\partial_t u_{\mathrm{imp}}(t_i,\bX_i)
&-\Big(\Delta u_{\mathrm{imp}}(t_i,\bX_i)
      -2H(\bX_i)\,\nabla u_{\mathrm{imp}}(t_i,\bX_i)\cdot \bn(\bX_i) \nonumber
     \\& -\bn(\bX_i)^\top(\nabla^2 u_{\mathrm{imp}}(t_i,\bX_i))\bn(\bX_i)\Big)
= f(t_i,\bX_i), \qquad (t_i, \bX_i)\in \mathring P_t\\
&\beta\big(u_{\mathrm{imp}}(0,\bX_i)-u_0(\bX_i)\big)=0,\qquad \bX_i\in \mathring P_t^0.
\end{align}
\end{subequations}
Since \eqref{eq:rann-st-imp} is linear in $\bc$, \eqref{eq:heat-colloc-imp} is solved in the least-squares sense for $\bc$.

\smallskip
\noindent\textbf{Point-cloud-based approach.}
If $\Gamma$ is given only by a point cloud $P_\Gamma=\{\bX_m\}_{m=1}^{N_x}\subset\Gamma$, we use the same space--time
ansatz \eqref{eq:rann-st-imp} and enforce \eqref{eq:heat-colloc-imp} on $\mathring P_t=P_T\times P_\Gamma$.
The difference is that the geometric quantities $\bn(\bX_m)$ and $H(\bX_m)$ are reconstructed from the point cloud,
e.g.\ via local quadratic fitting \cite{Cazals2005fitting} and the first and second fundamental forms \cite{Walker2015fs}.

\begin{remark}[Linear vs.\ nonlinear time-dependent PDEs]
In this paper, we focus on linear surface PDEs, for which the training objective leads to a linear least-squares problem
in the output weights. Nonlinear surface PDEs can be treated by solving a nonlinear least-squares problem or by applying
linearization schemes such as Picard or Newton iterations; see
\cite{Dong2021locELM, Shang2022DeepPetrov, Shang2024DeepPetrov, Sun2025lrnndgkdv}.
\end{remark}

\section{RaNN methods for PDEs on evolving surfaces}
\label{sec:evolving}

Solving PDEs on evolving surfaces is of considerable importance in many applications.  
In this section, we present a preliminary application of the proposed RaNN framework to the simplest linear problem on evolving surfaces.

In traditional mesh-based numerical methods, the mesh must be updated at each time step, which substantially increases computational complexity. The method proposed in this work is essentially a mesh-free method, which allows the entire set of points on the evolving surface across the spatiotemporal domain to be obtained in a single computation. { In contrast to traditional methods that require interpolation to update the mesh, the proposed approach stores the surface information directly as network parameters. This allows for subsequent usage through simple algebraic operations, eliminating the need for complex interpolation steps.}

\subsection{RaNN methods for capturing surface evolution}\label{subsec:surface-evolution}

Before solving a PDE on an evolving surface $\Gamma(t)$, we first approximate the surface motion.
Let $\Gamma_0$ be the initial surface and let $\bv:\mathbb R\times\mathbb R^3\to\mathbb R^3$ be a prescribed velocity field.
Assume that the motion is described by a flow map
\[
  \bx:[0,T]\times \Gamma_0\to\mathbb R^3,\qquad (t,\bX_0)\mapsto \bx(t,\bX_0),
\]
which satisfies the surface-trajectory ODE
\begin{subequations}\label{eq:evol-surface-flow}
\begin{align}
\partial_t \bx(t,\bX_0) &= \bv\big(t,\bx(t,\bX_0)\big), \qquad t\in[0,T],\ \bX_0\in\Gamma_0,\\
\bx(0,\bX_0) &= \bX_0, \qquad \bX_0\in\Gamma_0.
\end{align}
\end{subequations}
For clarity, we assume that $\Gamma_0$ is parametrized by $\bX_0(\xi,\eta)$ on a parameter domain $D$ so that
$\Gamma_0=\{\bX_0(\xi,\eta):(\xi,\eta)\in D\}$; however, the construction below is formulated directly in the embedded
coordinates $\bX_0\in\Gamma_0$.

Writing $\bx=(x_1,x_2,x_3)^\top$ and $\bv=(v_1,v_2,v_3)^\top$, \eqref{eq:evol-surface-flow} is equivalent to the coupled system
\begin{align}\label{eq:evol-components}
\partial_t x_i(t,\bX_0) = v_i\!\big(t,x_1(t,\bX_0),x_2(t,\bX_0),x_3(t,\bX_0)\big),\qquad i=1,2,3,
\end{align}
with $x_i(0,\bX_0)=(\bX_0)_i$.

We introduce three randomized neural networks $\mathcal N_e^1,\mathcal N_e^2,\mathcal N_e^3$ to approximate
$x_1,x_2,x_3$, respectively. Each network takes as input the time $t$ and the initial point $\bX_0\in\Gamma_0$
and outputs the corresponding coordinate at time $t$:
\begin{equation}\label{eq:rann-flow-3nets}
  \mathcal N_e^i(t,\bX_0)=\bc^{\,i}\,\rho\!\left(\bw^{\,i}\begin{bmatrix} t\\ \bX_0 \end{bmatrix}+\bb^{\,i}\right),
  \qquad i=1,2,3,
\end{equation}
where $\bw^{\,i}\in\mathbb R^{M\times 4}$ is sampled randomly, $\bc^{\,i}\in\mathbb R^{1\times M}$ is trainable, and
$\rho$ is applied componentwise.

Following the parameter selection strategy introduced earlier, we use a shallow randomized network with separated
time and spatial weights. Specifically, in \eqref{eq:rann-flow-3nets} we write $\bw^{\,i}=[\bw^{\,i}_t,\bw^{\,i}_x]$ with
$\bw^{\,i}_t\in\mathbb R^{M\times 1}$ and $\bw^{\,i}_x\in\mathbb R^{M\times 3}$ drawn independently from
$U(-r_t^i,r_t^i)$ and $U(-r_x^i,r_x^i)$, respectively.
The biases are defined as
\[
  \bb^{\,i}:=-(\bw^{\,i}\odot (\bB^{\,i})^\top)\,\mathbf 1_4,\qquad \mathbf 1_4=(1,1,1,1)^\top,
\]
where $\bB^{\,i}\in\mathbb R^{4\times M}$ is sampled to cover the range of $I\times\widetilde D$, with $\widetilde D$ an
axis-aligned bounding box of $\Gamma_0$.

Let $P_0=\{\bX_{0,m}\}_{m=1}^{N_x}\subset\Gamma_0$ be collocation points on the initial surface and
$P_T=\{t_n\}_{n=1}^{N_t}\subset[0,T]$ be time collocation points. Define the space--time set
$\mathring P:=P_T\times P_0$ and the initial set $\mathring P^0:=\{0\}\times P_0$. We determine the trainable output weights $\{\bC^{\,i}\}_{i=1}^3$ by enforcing \eqref{eq:evol-components} at space--time collocation points, specifically, we seek $\{\mathcal N_e^i\}_{i=1}^3$ such that
\begin{subequations}\label{eq:flow-colloc-3nets}
\begin{align}
&\partial_t \mathcal N_e^i(t_n,\bX_{0,m})
= v_i\!\big(t_n,\mathcal N_e^1(t_n,\bX_{0,m}),\mathcal N_e^2(t_n,\bX_{0,m}),\mathcal N_e^3(t_n,\bX_{0,m})\big),
\quad i=1,2,3,\ \forall (t_n,\bX_{0,m})\in\mathring P,\\
&\beta\Big(\big(\mathcal N_e^1,\mathcal N_e^2,\mathcal N_e^3\big)(0,\bX_{0,m})-\bX_{0,m}\Big)=0,
\quad \forall \bX_{0,m}\in P_0,
\end{align}
\end{subequations}
and solve the resulting system in the least-squares sense, where $\beta>0$ controls the enforcement of the initial condition.
In general, because $\bv$ depends on the network outputs, \eqref{eq:flow-colloc-3nets} is a \emph{nonlinear} least-squares problem.
It can be solved by standard nonlinear least-squares solvers or linearization strategies \cite{Sun2025lrnndgkdv} (e.g., Gauss--Newton).

\begin{remark}[On topology preservation]
The formulation above assumes that the flow map $\bx(t,\cdot)$ transports $\Gamma_0$ to $\Gamma(t)$ without topological changes,
i.e., $\Gamma(t)$ remains homeomorphic and typically diffeomorphic to $\Gamma_0$ for all $t\in[0,T]$.
Capturing evolutions with topology changes (e.g., pinch-off or merging) requires a different representation, and will be
investigated in future work.
\end{remark}

\subsection{RaNN Methods for solving PDEs}\label{subsec:evol-pde}

Once an approximation of the surface evolution is available via the learned flow map $\mathcal N_e$,
we can sample collocation points on the space--time surface $I\times\Gamma(t)$.
We now extend the proposed RaNN methodology to the advection--diffusion equation on evolving surfaces:
\begin{subequations}\label{eq:ev-u}
\begin{align}
\partial^\bullet u + u\,\nabla_{\Gamma(t)}\cdot \bv - \Delta_{\Gamma(t)}u &= f
\quad \text{on } \Gamma(t),\ t\in[0,T],\\
u(0,\cdot) &= u_0 \quad \text{on } \Gamma_0.
\end{align}
\end{subequations}
Here $\partial^\bullet u$ denotes the material derivative. Since $u$ is defined only on $\Gamma(t)$,
we interpret $\partial^\bullet u$ through an extension $\tilde u$ in a neighborhood of $\Gamma(t)$:
\[
\partial^\bullet u := \partial_t \tilde u + \bv\cdot\nabla \tilde u \quad \text{on } \Gamma(t).
\]
In our numerical formulation we take a RaNN to approximate the solution, so all ambient derivatives are computed directly
from the network.

Let $\mathcal N_e=(\mathcal N_e^1,\mathcal N_e^2,\mathcal N_e^3)^\top$ denote the trained flow-map network obtained in
Section~\ref{subsec:surface-evolution}. Then, for $(\xi_1,\xi_2)\in D$, the evolving surface is represented by the
time-dependent parametrization
\[
\bX(t,\xi,\eta) := \mathcal N_e\!\big(t,\bX_0(\xi,\eta)\big)\in\Gamma(t).
\]
By differentiating $\bX(t,\xi_1,\xi_2)$ with respect to $(\xi_1,\xi_2)$, we obtain the first and second fundamental forms,
and hence approximations of the unit normal vector $\bn_{\mathcal N}(t,\xi_1,\xi_2)$ and the mean curvature
$H_{\mathcal N}(t,\xi_1,\xi_2)$. These geometric quantities are used to evaluate
$\nabla_{\Gamma(t)}\cdot\bv$ and $\Delta_{\Gamma(t)}$ at collocation points.

Choose parameter samples $\{(\xi_{1,m},\xi_{2,m})\}_{m=1}^{N_x}\subset D$ and time samples $\{t_n\}_{n=1}^{N_t}\subset[0,T]$.
Define the space--time collocation set on the evolving surface
\[
P_e := \Big\{(t_n,\bX(t_n,\xi_{1,m},\xi_{2,m})):\ n=1,\dots,N_t,\ m=1,\dots,N_x\Big\}\subset [0,T]\times\Gamma(t),
\]
and the initial set
\[
P_e^0 := \Big\{(0,\bX_0(\xi_{1,m},\xi_{2,m})):\ m=1,\dots,N_x\Big\}\subset \{0\}\times\Gamma_0.
\]

We introduce an ambient space--time RaNN
\[
u_{e}(t,\bX)=\bc\,\rho\!\left(\bw\begin{bmatrix}t\\ \bX\end{bmatrix}+\bb\right),
\qquad (t,\bX)\in[0,T]\times\mathbb R^3,
\]
with $\bw=[\bw_t,\bw_x]\in\mathbb R^{M\times 4}$, where $\bw_t\in\mathbb R^{M\times 1}$ and
$\bw_x\in\mathbb R^{M\times 3}$ are drawn from $U(-r_t,r_t)$ and $U(-r_x,r_x)$, respectively, and
$\bb:=-(\bw\odot\bB^\top)\mathbf 1_4$ with $\bB$ sampled on a bounding box of $[0,T]\times\Gamma_0$ as before.

At each collocation point $(t_i,\bX_i)\in P_e$, we evaluate the operator
\[
\mathcal L  u_{e}
:= \partial^\bullet  u_{e} +  u_{e}\,(\nabla_{\Gamma}\cdot\bv) - \Delta_{\Gamma} u_{e},
\]
where $\partial^\bullet  u_{e}=\partial_t u_{e}+\bv\cdot\nabla u_{e}$.
Using the embedding identities with $\bn=\bn_{\mathcal N}$ and $H=H_{\mathcal N}$, we compute
\[
\nabla_{\Gamma}\cdot\bv
= \nabla\cdot\bv - \bn_{\mathcal N}^\top(\nabla\bv)\bn_{\mathcal N},
\qquad
\Delta_{\Gamma} u_{e}
= \Delta u_{e} - 2H_{\mathcal N}\,\nabla u_{e}\cdot \bn_{\mathcal N}
- \bn_{\mathcal N}^\top(\nabla^2 u_{e})\bn_{\mathcal N}.
\]
We then enforce the PDE and the initial condition by collocation:
\begin{subequations}\label{eq:ev-uNN}
\begin{align}
\mathcal L  u_e(t_i,\bX_i) &= f(t_i,\bX_i), \qquad &&\forall (t_i,\bX_i)\in P_e,\\
 \beta\big(u_e(0,\bX_i)-u_0(\bX_i)\big)&=0,\qquad &&\forall (0,\bX_i)\in P_e^0,
\end{align}
\end{subequations}
and solve the resulting system by the least-squares method for the output weights $\bc$.

\section{Numerical Experiments}
\label{sec:ex}

In this section, RaNN is applied to various problems on static and evolving surfaces in order to assess its performance.
All experiments were conducted on an Intel(R) Xeon(R) Gold 6240 CPU @ 2.60GHz.
Neural networks were implemented in Python using PyTorch (version 1.12.1) with a fixed random seed to ensure reproducibility.
The output-layer parameters were obtained by solving a least-squares problem using SciPy (version 1.4.1).
The hidden-layer weights were sampled i.i.d.\ from uniform distributions $U(-r_x,r_x)$ for spatial inputs and
$U(-r_t,r_t)$ for the temporal input (when present), and were kept fixed throughout training; only the output weights were optimized.
For further discussion on the impact and practical selection of $r_x$ and $r_t$, see
\cite{Dong2021locELMW, Zhang2024Transferable, AGRANN2025}.
The activation function used in all examples was the hyperbolic tangent $\tanh(\cdot)$.

To evaluate accuracy, we randomly select $N_{\rm test}$ test points $\{\bx_i\}_{i=1}^{N_{\rm test}}$ on the surface and compute
the relative $\ell^2$ error between the exact solution $u^*$ and the numerical solution $u$:
\[
E^{\ell^2}(u)
:= \frac{\left(\sum_{i=1}^{N_{\rm test}} \big(u^*(\bx_i)-u(\bx_i)\big)^2\right)^{1/2}}
{\left(\sum_{i=1}^{N_{\rm test}} \big(u^*(\bx_i)\big)^2\right)^{1/2}} .
\]

Unless otherwise stated, test points are drawn independently of the collocation (training) points.

Since we employ shallow neural networks with fixed and known input-to-hidden weights, the required derivatives of the network output
with respect to the inputs (including second-order derivatives for Laplace--Beltrami terms) can be computed in closed form.
This avoids automatic differentiation in the residual evaluation and reduces computational overhead in our implementation.

\begin{example}[Laplace--Beltrami equation on a torus surface] 
\label{ex_ti_par}
We consider the Laplace--Beltrami equation
\[
-\Delta_\Gamma u = f \quad \text{on } \Gamma,
\]
where the torus $\Gamma \subset \mathbb{R}^3$ admits a global parametrization $X=X(\xi_1,\xi_2)$:
\[
\left\{
\begin{array}{rcl}
x &=& \left( \frac{1}{4} \cos(\xi_1) + 1 \right) \cos(\xi_2), \\
y &=& \left( \frac{1}{4} \cos(\xi_1) + 1 \right) \sin(\xi_2), \\
z &=& \frac{1}{4} \sin(\xi_1),
\end{array}
\right.
\qquad (\xi_1,\xi_2)\in[0,2\pi]\times[0,2\pi].
\]
We set $u(x,y,z)=\sin(x)\exp(\cos(y-z))$, and derive the source term $f$ from this exact solution.
\end{example}

We first solve this problem using the parametrization-based formulation on the single chart $D=[0,2\pi]^2$.
Since $\Gamma$ admits a global parametrization, we take $J=1$ and approximate $u$ by a single RaNN defined on $D$.
Because the parametrization is periodic in both $\xi_1$ and $\xi_2$, we additionally enforce periodic compatibility across the
identified boundary pairs of $D$ (matching function values and first derivatives across $\xi_1=0$ vs.\ $\xi_1=2\pi$ and
$\xi_2=0$ vs.\ $\xi_2=2\pi$) through boundary collocation.
For closed-surface Poisson problems, all errors are computed after removing the constant mode by shifting the numerical solution to match a reference value at a fixed point (cf.\ Remark~\ref{rem:const-shift}).

The input-to-hidden weights are sampled from $U(-1,1)$, and the bias sampling set $\bB$ is drawn from $U(\tilde D)$ with
$\tilde D=[0,2\pi]^2$, consistent with the parameter ranges.
We uniformly sample the parameter domain $D=[0,2\pi]^2$ with $N_0$ points per direction, yielding
\[
N_{\rm col}=N_0^2
\]
interior collocation points for enforcing the PDE. Additional collocation points are placed on the boundary pairs to impose the
periodic compatibility conditions.
Table~\ref{tablel_torus_phi} reports the relative errors and training times for different network widths $M$ and collocation sizes
$N_{\rm col}$. Non-monotone behavior can occur due to the random-feature draw and the conditioning of the least-squares system.

\begin{table}[H]
\centering
\begin{tabular}{ccccccc}
\toprule[1.5pt]
$M$     & \multicolumn{2}{c}{600} & \multicolumn{2}{c}{1000} & \multicolumn{2}{c}{1400} \\ \cmidrule[0.5pt](l){2-3} \cmidrule[0.5pt](l){4-5} \cmidrule[0.5pt](l){6-7} 
$N_{\rm col}$   & Error      & Time (s)    & Error       & Time (s)    & Error       & Time (s)    \\ \midrule[0.5pt]
900             & 1.77E-04   & 0.078   & 2.95E-04    & 0.139   & 4.17E-04    & 0.191   \\
2500            & 4.77E-05   & 0.111   & 4.04E-06    & 0.213   & 1.63E-06    & 0.304   \\
4900           & 1.79E-05   & 0.172   & 2.06E-07    & 0.449   & 1.32E-07    & 0.721   \\
8100           & 3.86E-05   & 0.415   & 3.66E-07    & 0.666   & 5.90E-08    & 1.018  \\
\bottomrule[1.5pt]
\end{tabular}
\caption{Relative errors obtained by the single-chart parametrization-based approach
(with periodic boundary compatibility enforced) in Example~\ref{ex_ti_par}}
\label{tablel_torus_phi}
\end{table}

These results already surpass those reported in \cite{Hu2024pinn}.
To further improve performance, we employ parametrization-based approach(no separate boundary enforcement) \eqref{no_boundary}, which naturally incorporates periodic boundary conditions.
Here $\bB$ is sampled from $U([-1.5,1.5]\times[-1.5,1.5]\times[-0.5,0.5])$, while keeping all other parameters unchanged.
The updated results are shown in Table~\ref{tablel_torus_phixyz}.

\begin{table}[H]
\centering
\begin{tabular}{ccccccc}
\toprule[1.5pt]
$M$     & \multicolumn{2}{c}{600} & \multicolumn{2}{c}{1000} & \multicolumn{2}{c}{1400} \\ \cmidrule[0.5pt](l){2-3} \cmidrule[0.5pt](l){4-5} \cmidrule[0.5pt](l){6-7} 
$N_{\rm col}$    & Error      & Time (s)    & Error       & Time (s)    & Error       & Time (s)    \\ \midrule[0.5pt]
900           & 4.23E-07   & 0.081   & 1.33E-07    & 0.143   & 1.01E-07    & 0.178   \\
2500           & 4.16E-11   & 0.118   & 4.74E-12    & 0.229   & 3.34E-12    & 0.400   \\
4900          & 1.01E-10   & 0.262   & 1.93E-12    & 0.548   & 4.87E-14    & 1.015   \\
8100          & 7.61E-11   & 0.597   & 7.11E-13    & 1.072   & 2.57E-13    & 1.626   \\
\bottomrule[1.5pt]
\end{tabular}
\caption{Relative errors obtained by parametrization-based approach (no separate boundary enforcement) in Example \ref{ex_ti_par}.}
\label{tablel_torus_phixyz}
\end{table}

\begin{example}[Laplace--Beltrami equation on a Cheese-like surface] 
\label{ex_ti_phi}
We consider the Laplace--Beltrami equation $-\Delta_\Gamma u=f$ on a static surface $\Gamma$ with the same exact solution
$u(x,y,z)=\sin(x)\exp(\cos(y-z))$. The surface does not admit a convenient global parametrization, but is represented implicitly as
the zero level set $\Gamma=\{\phi_{CL}=0\}$ with
\begin{align*}
\phi_{CL}(x,y,z)=&(4x^2-1)^2+(4y^2-1)^2+(4z^2-1)^2+16(x^2+y^2-1)^2\\
&+16(x^2+z^2-1)^2+16(y^2+z^2-1)^2-16.
\end{align*}
The source term $f$ is derived from the exact solution.
\end{example}

To generate collocation points on $\Gamma$, we employ a signed-distance-based meshing strategy \cite{Persson2004distmesh}
using the DistMesh package (MATLAB), which yields a point set $P_{CL}^{\rm coarse}$ with $N_{\rm all}=21192$ samples.
We set $r_x=1$ and choose the bounding box $\widetilde D=[-1.5,1.5]^3$ according to the geometry of the surface.
In each run, we use a network width $M$ and randomly select $N_{\rm col}$ points from $P_{CL}^{\rm coarse}$ for training,
while evaluating errors on the full set $P_{CL}^{\rm coarse}$.
The results are reported in Table~\ref{tablel_cl_coarse}.
Note that the test set $P_{CL}^{\rm coarse}$ contains the randomly selected training subset; since $N_{\rm all}\gg N_{\rm col}$ and we focus on global accuracy over the entire surface, this overlap has a negligible impact on the reported errors.

\begin{table}[H]
\centering
\begin{tabular}{ccccccc}
\toprule[1.5pt]
$M$   & \multicolumn{2}{c}{1200} & \multicolumn{2}{c}{1600} & \multicolumn{2}{c}{2000} \\ \cmidrule[0.5pt](l){2-3} \cmidrule[0.5pt](l){4-5} \cmidrule[0.5pt](l){6-7} 
$N_{\rm col}$   & Error       & Time (s)    & Error       & Time (s)    & Error       & Time (s)    \\ \midrule[0.5pt]
2000          & 5.70E-07    & 0.229   & 1.18E-06    & 0.366   & 7.98E-08    & 0.485   \\
4000          & 1.16E-08    & 0.334   & 1.25E-06    & 0.740   & 6.42E-08    & 0.984   \\
6000          & 5.57E-07    & 0.567   & 3.01E-07    & 0.988   & 3.51E-08    & 1.387   \\
8000           & 6.43E-08    & 0.770   & 9.04E-08    & 1.344   & 1.29E-08    & 1.755   \\
\bottomrule[1.5pt]
\end{tabular}
\caption{Relative errors (evaluated on $P_{CL}^{\rm coarse}$) obtained by the level-set-based approach in Example~\ref{ex_ti_phi}.}
\label{tablel_cl_coarse}
\end{table}

\begin{table}[H]
\centering
\begin{tabular}{ccccccc}
\toprule[1.5pt]
$M$     & \multicolumn{2}{c}{1200} & \multicolumn{2}{c}{1600} & \multicolumn{2}{c}{2000} \\ \cmidrule[0.5pt](l){2-3} \cmidrule[0.5pt](l){4-5} \cmidrule[0.5pt](l){6-7} 
$N_{\rm col}$    & Error       & Time (s)    & Error       & Time (s)    & Error       & Time (s)    \\ \midrule[0.5pt]
2000          & 1.44E-09    & 0.236   & 6.93E-10    & 0.331   & 1.36E-10    & 0.456   \\
4000           & 1.04E-09    & 0.427   & 1.08E-10    & 0.606   & 3.72E-12    & 0.991   \\
6000           & 3.30E-10    & 0.690   & 1.83E-11    & 1.013   & 3.60E-12    & 1.285   \\
8000           & 3.26E-10    & 0.900   & 1.19E-11    & 1.308   & 1.76E-12    & 1.702   \\
\bottomrule[1.5pt]
\end{tabular}
\caption{Relative errors (evaluated on $P_{CL}^{\rm fine}$) obtained by the level-set-based approach in Example~\ref{ex_ti_phi}.}
\label{tablel_cl_fine}
\end{table}

We observe that further accuracy improvement is limited when training points deviate from the zero level set.
To quantify the geometric sampling error of a point set $P=\{\bX_i\}_{i=1}^{N_{\rm all}}$, we define the RMS level-set residual
\[
E_p(P):=\left(\frac1{N_{\rm all}}\sum_{i=1}^{N_{\rm all}}|\phi_{CL}(\bX_i)|^2\right)^{1/2}.
\]
For $P_{CL}^{\rm coarse}$ we obtain $E_p(P_{CL}^{\rm coarse})=1.17\times 10^{-5}$.

\begin{figure}[H] 
	\centering  
	\begin{subfigure}{0.45\textwidth}
		\centering      
		\includegraphics[width=1.0\textwidth]{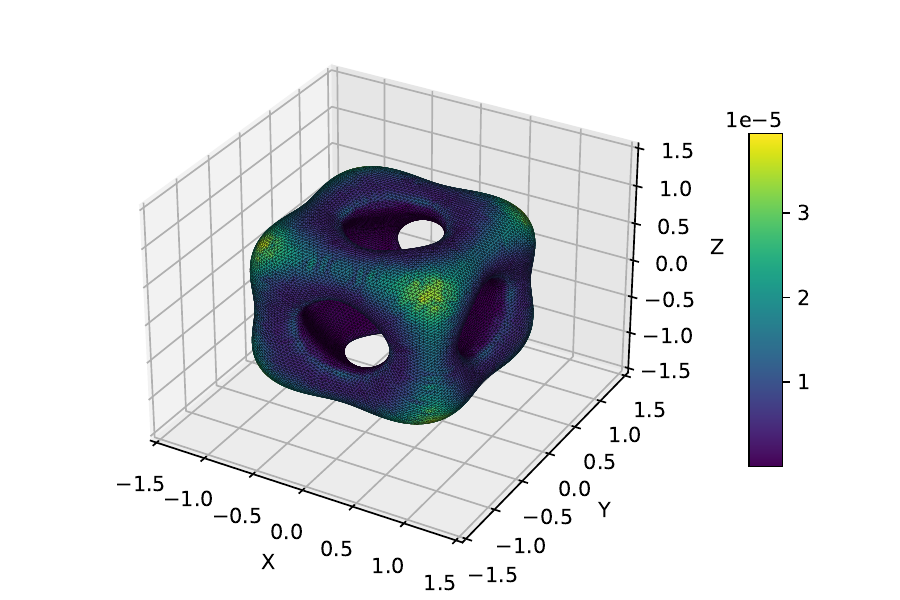}  
		\caption{Errors of the point set $P_{CL}^{coarse}$.}
		%\label{fig:a}
	\end{subfigure}
	\begin{subfigure}{0.45\textwidth}
		\centering      
		\includegraphics[width=1.0\textwidth]{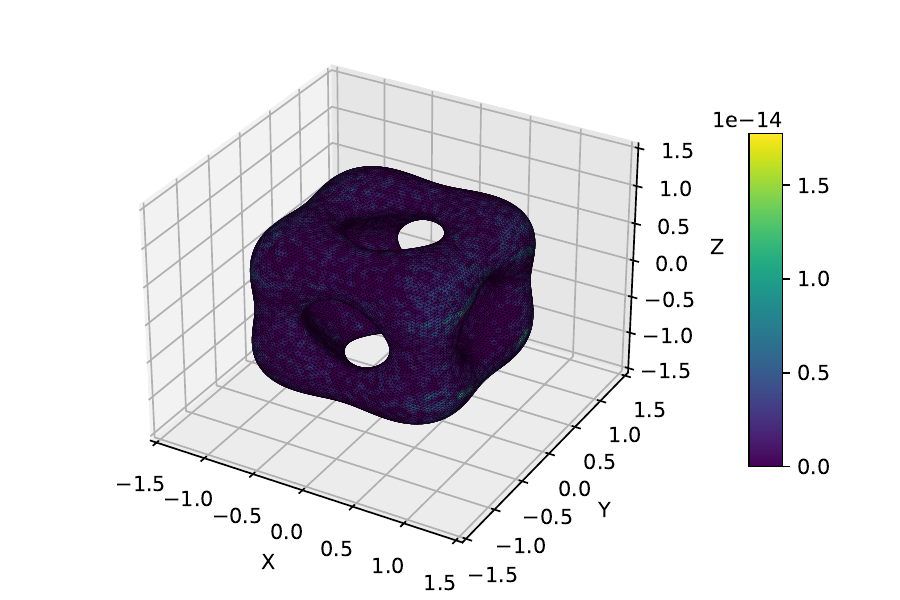}  
		\caption{Errors of the point set $P_{CL}^{fine}$.}
		%\label{fig:b}
	\end{subfigure}
	\caption{Pointwise absolute errors of the two point sets.}  
	\label{Ep_error}
\end{figure}

To reduce the off-surface sampling error, we project each coarse point onto the zero level set $\{\phi_{CL}=0\}$
by a few Newton-type updates along the level-set gradient, until $|\phi_{CL}(\bX)|$ falls below a prescribed tolerance.
The resulting point set is denoted by $P_{CL}^{\rm fine}$.
To mitigate this issue, we employ the \texttt{fsolve} function in SciPy to adjust the coarse points and obtain corrected coordinates, denoted as $P_{CL}^{fine}$ with $E_p(P_{CL}^{\rm fine})=1.87\times 10^{-15}$.
Figure~\ref{Ep_error} shows the pointwise residual $|\phi_{CL}(\bX_i)|$ for both point sets.

Keeping all other settings unchanged, we draw training points from $P_{CL}^{\rm fine}$ and report the corresponding results in
Table~\ref{tablel_cl_fine}, demonstrating a substantial accuracy improvement at comparable computational cost.

\begin{example}[Heat equation on the cheese-like surface]\label{ex_td_cl}
We consider the heat equation on the static cheese-like surface $\Gamma=\{\phi_{CL}=0\}$ over the time interval $I=[0,1]$.
The exact solution is
\[
u(t,x,y,z)=\sin\!\big(x+\sin t\big)\exp(\cos(y-z)),
\]
from which the source term is derived.
\end{example}

As in Example~\ref{ex_ti_phi}, we use two spatial point sets on $\Gamma$: a coarse set $P^{\rm coarse}_{CL}$ and a refined set
$P^{\rm fine}_{CL}$, both of size $N_{\rm all}$.
We uniformly partition $I$ into $200$ subintervals and define the temporal set
\[
P_T=\{0=t_0<t_1<\cdots<t_{200}=1\}.
\]
The corresponding spatiotemporal point sets are
\[
\mathring P^{\rm coarse}_{CL}:=P_T\times P^{\rm coarse}_{CL},\qquad
\mathring P^{\rm fine}_{CL}:=P_T\times P^{\rm fine}_{CL}.
\]
In each run, we randomly select $N_{\rm col}$ points from $\mathring P^{\rm coarse}_{CL}$ (or $\mathring P^{\rm fine}_{CL}$)
to enforce the PDE residual, and select $N_0$ points from $\{0\}\times P^{\rm coarse}_{CL}$ (or $\{0\}\times P^{\rm fine}_{CL}$)
to impose the initial condition. The sets $\{1\}\times P^{\rm coarse}_{CL}$ and $\{1\}\times P^{\rm fine}_{CL}$ are used for evaluation.

\begin{table}[H]
\centering
\begin{tabular}{ccccccccc}
\toprule[1.5pt]
$M$    & \multicolumn{2}{c}{1200} & \multicolumn{2}{c}{1600} & \multicolumn{2}{c}{2000} & \multicolumn{2}{c}{2400} \\ \cmidrule[0.5pt](l){2-3} \cmidrule[0.5pt](l){4-5} \cmidrule[0.5pt](l){6-7} \cmidrule[0.5pt](l){8-9}
$N_{\rm col},\ N_0$ & Error       & Time (s)    & Error       & Time (s)    & Error       & Time (s)    & Error       & Time (s)    \\ \midrule[0.5pt]
2000, 1000        & 1.89E-04    & 0.280   & 1.40E-04    & 0.428   & 1.56E-04    & 0.655   & 1.45E-04    & 1.137   \\
4000, 1000       & 9.97E-05    & 0.642   & 4.94E-05    & 0.901   & 2.11E-05    & 1.204   & 1.21E-05    & 1.598   \\
6000, 1000        & 4.55E-05    & 0.915   & 1.89E-05    & 1.285   & 1.36E-05    & 1.571   & 5.90E-06    & 2.098   \\
8000, 1000        & 4.80E-05    & 1.095   & 1.52E-05    & 1.587   & 9.78E-06    & 1.985   & 3.43E-06    & 2.541   \\
\bottomrule[1.5pt]
\end{tabular}
\caption{Relative errors obtained by the level-set-based approach using $\mathring P^{\rm coarse}_{CL}$ in Example~\ref{ex_td_cl}.}
\label{tablel_cl_heat_coarse}
\end{table}

\begin{table}[H]
\centering
\begin{tabular}{ccccccccc}
\toprule[1.5pt]
$M$    & \multicolumn{2}{c}{1200} & \multicolumn{2}{c}{1600} & \multicolumn{2}{c}{2000} & \multicolumn{2}{c}{2400} \\ \cmidrule[0.5pt](l){2-3} \cmidrule[0.5pt](l){4-5} \cmidrule[0.5pt](l){6-7} \cmidrule[0.5pt](l){8-9}
$N_{\rm col},\ N_0$ & Error       & Time (s)    & Error       & Time (s)    & Error       & Time (s)    & Error       & Time (s)    \\ \midrule[0.5pt]
2000, 1000        & 1.89E-04    & 0.271   & 1.37E-04    & 0.477   & 1.65E-04    & 0.667   & 1.50E-04    & 0.938   \\
4000, 1000        & 9.95E-05    & 0.534   & 4.94E-05    & 0.834   & 1.97E-05    & 1.281   & 1.13E-05    & 1.637   \\
6000, 1000        & 4.55E-05    & 0.927   & 1.91E-05    & 1.231   & 1.29E-05    & 1.672   & 5.23E-06    & 2.111   \\
8000, 1000        & 4.80E-05    & 1.114   & 1.52E-05    & 1.582   & 9.71E-06    & 2.050   & 3.57E-06    & 2.540   \\
\bottomrule[1.5pt]
\end{tabular}
\caption{Relative errors obtained by the level-set-based approach using $\mathring P^{\rm fine}_{CL}$ in Example~\ref{ex_td_cl}.}
\label{tablel_cl_heat_fine}
\end{table}

We set $r_x=r_t=0.6$, $\beta=100$, and sample $\bB\sim U([0,1]\times[-1.5,1.5]^3)$.
The results obtained with $\mathring P^{\rm coarse}_{CL}$ and $\mathring P^{\rm fine}_{CL}$ are reported in
Tables~\ref{tablel_cl_heat_coarse} and \ref{tablel_cl_heat_fine}, respectively.
For this time-dependent example and the present parameter choices, using coarse versus refined spatial points leads to
comparable accuracy, suggesting that the overall error is dominated by factors other than the geometric residual of the spatial sampling.
A plausible explanation is that the heat equation is parabolic and exhibits a smoothing effect in time, so small off-surface
perturbations of the collocation locations are damped rather than amplified.
Moreover, in our implementation the dominant error sources are likely the finite model capacity (fixed-width random features),
the stochasticity of collocation subsampling in the space--time set, and the conditioning of the resulting least-squares system.
Under the present choices of $(r_x,r_t)$, $\beta$, and the space--time sampling density, these factors appear to dominate the overall
error, leading to comparable accuracies for $\mathring P^{\rm coarse}_{CL}$ and $\mathring P^{\rm fine}_{CL}$.

\begin{example}[Heat equation on a cup-shaped surface]\label{ex:cup_heat}
We solve the heat equation on a (non-closed) cup-shaped surface $\Gamma_{\rm cup}$ over $I=[0,1]$:
\[
\partial_t u - \alpha\,\Delta_{\Gamma}u = f \quad \text{on } I\times \Gamma_{\rm cup},
\]
with thermal diffusivity $\alpha=10$ and initial condition $u(0,\cdot)\equiv 10$ on $\Gamma_{\rm cup}$.
The surface $\Gamma_{\rm cup}$ consists of two parametrized parts, $\Gamma_{\rm cup}=\Gamma_{\rm cup}^1\cup\Gamma_{\rm cup}^2$,
defined by the charts $X_1:D_1\to\Gamma_{\rm cup}^1$ and $X_2:D_2\to\Gamma_{\rm cup}^2$:
\[
X_1(\xi_1,\xi_2)=
\Big(\sqrt{1-\big(\tfrac{\xi_2+1}{0.1}\big)^2}\cos\xi_1,\ \sqrt{1-\big(\tfrac{\xi_2+1}{0.1}\big)^2}\sin\xi_1,\ \xi_2\Big),
\quad D_1=[0,2\pi]\times[-1.1,-1],
\]
\[
X_2(\xi_1,\xi_2)=(\cos\xi_1,\ \sin\xi_1,\ \xi_2),
\quad D_2=[0,2\pi]\times[-1,1].
\]
A localized heat source is applied near the bottom for $t\in[0,0.5]$ and removed afterwards:
\[
f(t,\bX)=15\,\exp\!\big(-\big(x^2+y^2+(z+1.1)^2\big)\big)\,\mathbf 1_{[0,0.5]}(t).
\]
At the rim $\partial\Gamma_{\rm cup}=\{X_2(\xi_1,1):\xi_1\in[0,2\pi]\}$ we impose the Robin boundary condition
\[
u+\partial_\mu u = 10 \quad \text{on } I\times \partial\Gamma_{\rm cup},
\]
where $\mu$ denotes the outward unit \emph{co-normal} (tangent to $\Gamma_{\rm cup}$ and normal to $\partial\Gamma_{\rm cup}$).
\end{example}

To avoid explicitly enforcing compatibility across the junction $\Gamma_{\rm cup}^1\cap\Gamma_{\rm cup}^2$,
we use a \emph{single} space--time RaNN defined on embedded coordinates:
\[
\hat u(t,\bX)=\bc\,\rho(\bw_t\,t+\bw_x\,\bX+\bb),\qquad (t,\bX)\in I\times\mathbb R^3.
\]
We then evaluate $u$ on each chart via the pullbacks $u_i(t,\bxi)=\hat u(t,X_i(\bxi))$, $i=1,2$.
Since $\hat u$ is globally defined on $\Gamma_{\rm cup}$, the traces from the two charts coincide automatically along the
junction curve, so no separate interface mismatch enforcement is required.

We enforce the PDE residual at interior collocation points on $I\times D_1$ and $I\times D_2$, the initial condition at
$\{0\}\times(D_1\cup D_2)$, and the Robin boundary condition on $I\times\{\xi_2=1\}\subset I\times D_2$.
For the Robin term we use $\partial_\mu u=\nabla_\Gamma \hat u\cdot \mu$, where $\nabla_\Gamma \hat u=(\bI-\bn\bn^\top)\nabla \hat u$
and $\mu$ can be computed from the parametrization of $X_2$ on $\xi_2=1$.

In our experiments we set $r_t=r_x=1$, $\beta=100$, and use network width $M=2000$.
We take $N_t=N_{\xi_1}=N_{\xi_2}=25$ uniform samples in each coordinate direction:
$N_tN_{\xi_1}N_{\xi_2}$ interior space--time points per patch for the PDE residual,
$N_{\xi_1}N_{\xi_2}$ points per patch at $t=0$ for the initial condition,
and $N_tN_{\xi_1}$ points on $I\times\{\xi_2=1\}$ for the Robin boundary condition.
Figure~\ref{cup_heat} shows that the predicted temperature captures both the heating phase ($0\le t\le 0.5$) and the cooling phase
($0.5<t\le 1$).

When neither an analytical parametrization nor a level-set representation is available and the surface is given only as a point cloud,
we can apply the point-cloud-based RaNN approach introduced in Section~\ref{subsec:ls-pc} by reconstructing local geometry
(normals/curvatures) from the point set, and design heat experiments analogous to the present one to assess robustness.

\begin{figure}[H] 
	\centering  
	\begin{subfigure}{0.9\textwidth}
		\centering      
		\includegraphics[width=1.0\textwidth]{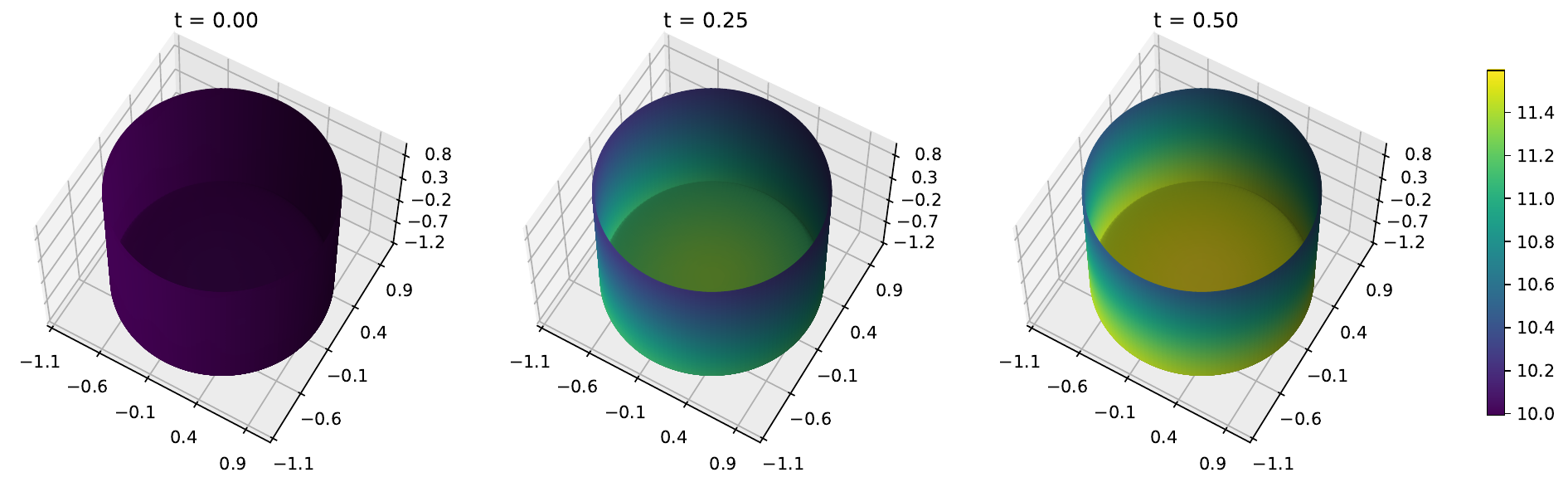}  
		\caption{Heating process ($0\le t \le 0.5$).}
		%\label{fig:a}
	\end{subfigure}\\
	\begin{subfigure}{0.9\textwidth}
		\centering      
		\includegraphics[width=1.0\textwidth]{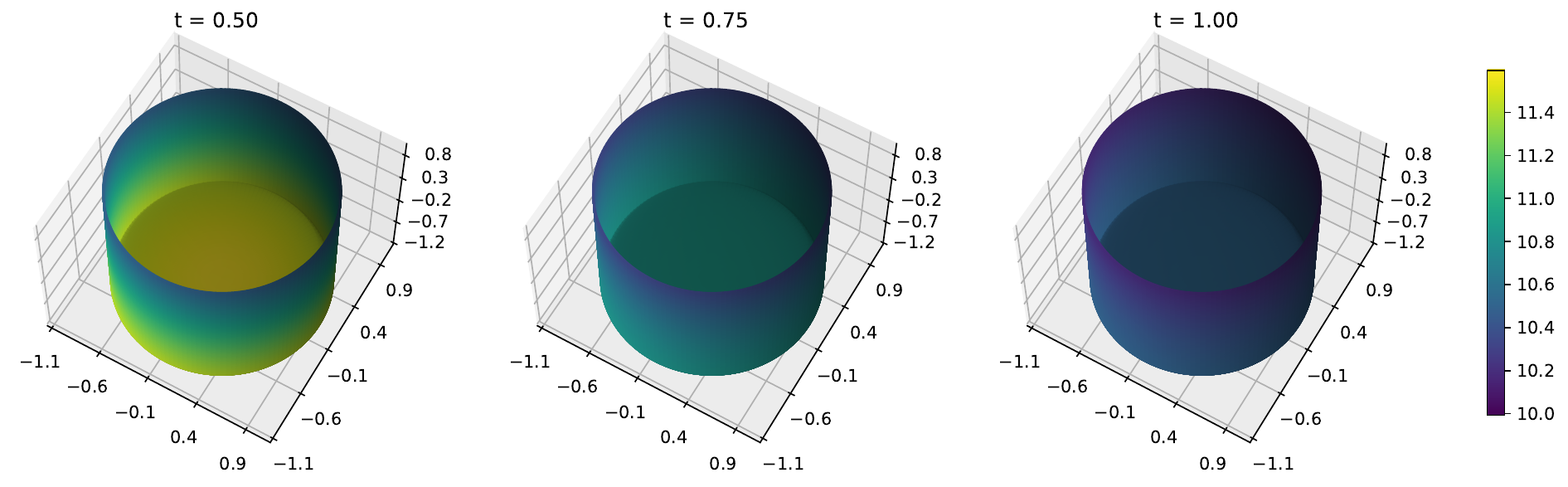}  
		\caption{Cooling process ($0.5 < t \le 1$).}
		%\label{fig:b}
	\end{subfigure}
	\caption{Predicted temperature distribution on the cup-shaped surface.}  
	\label{cup_heat}
\end{figure}

\begin{example}[Heat equation on a point-cloud surface]\label{ex_td_bunny}
We consider the heat equation on the bunny surface $\Gamma_b$ over $I=[0,1]$:
\[
\partial_t u - \alpha\,\Delta_\Gamma u = f \quad \text{on } I\times \Gamma_b,
\]
with thermal diffusivity $\alpha=0.015$ and initial condition $u(0,\cdot)\equiv 0$.
A localized heat source is applied near the bunny's tail at $\bX_\ast=(0.068,0.06,0.01)$:
\[
f(t,x,y,z)=10\,\exp\!\Big(-\big((100(x-0.068))^2+(100(y-0.06))^2+(100(z-0.01))^2\big)\Big),
\]
which is time-independent over $t\in[0,1]$.
\end{example}

The surface $\Gamma_b$ is given only as a point cloud.
Following the point-cloud-based approach described earlier, we reconstruct the local geometric information needed to evaluate
$\Delta_\Gamma$ (e.g., unit normals and mean curvature) via local quadratic fitting \cite{Cazals2005fitting,Walker2015fs}.
We then approximate $u$ by an ambient space--time RaNN with network width $M=2500$,
using $r_x=r_t=1$ and the bias sampling set $\bB\sim U([0,1]\times[-0.1,0.07]\times[0,0.2]\times[-0.07,0.07])$.
The initial condition is enforced by a penalty parameter $\beta=50$.

For training, we randomly sample $N_{\rm col}=40000$ spatiotemporal collocation points from $[0,1]\times \Gamma_b$ to enforce the PDE
residual, and $N_0=200$ points from $\{0\}\times\Gamma_b$ to impose the initial condition.
The predicted temperature distribution is shown in Figure~\ref{figure_bunny_heat_apply}, the temperature rise originates from the
vicinity of the tail, consistent with the expected physical behavior.

\begin{figure}[H]  
	\centering   	
	\includegraphics[width=0.8\textwidth]{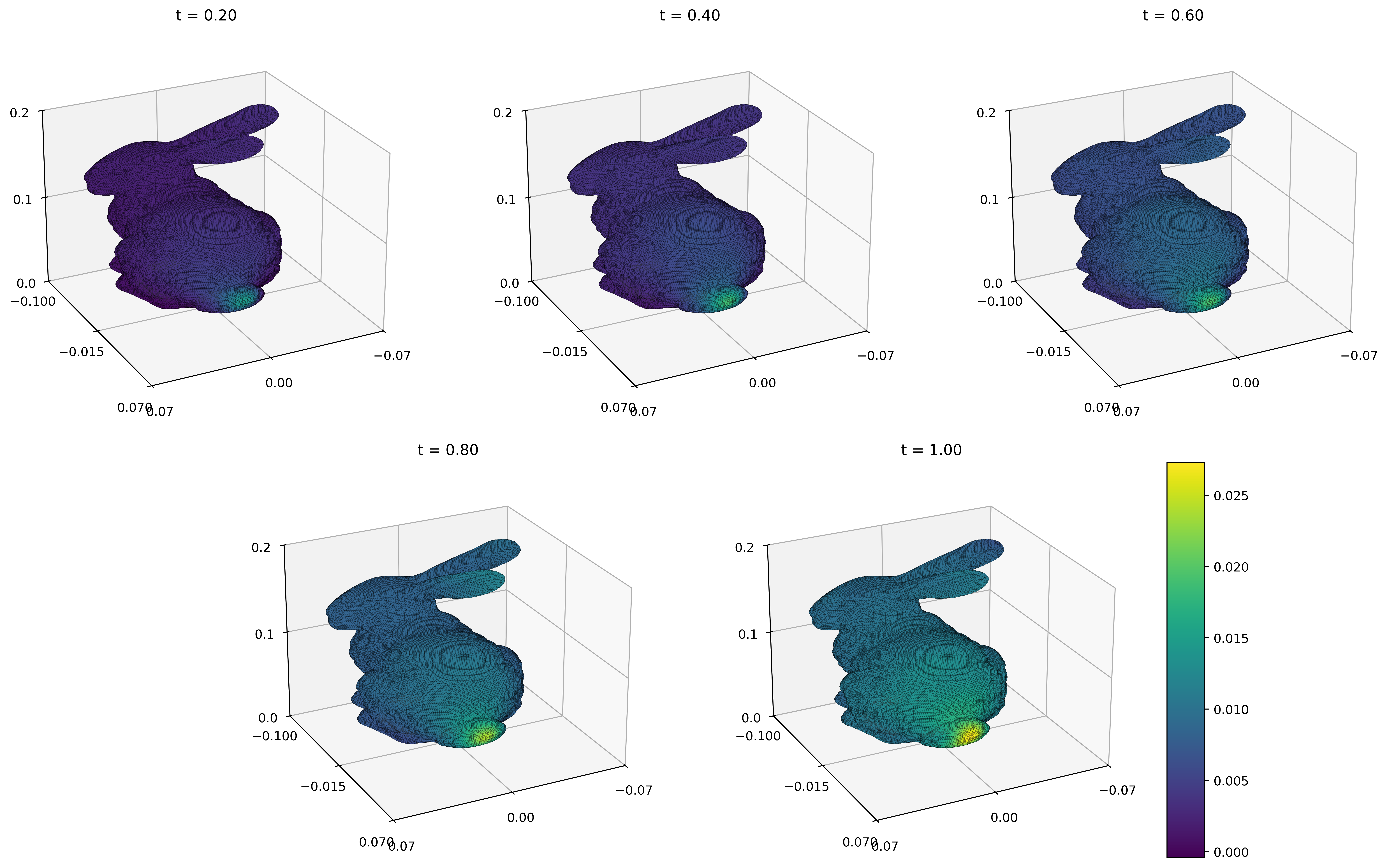}   
	\caption{Predicted heat distribution on the bunny surface.} 
	\label{figure_bunny_heat_apply}
\end{figure}

\begin{example}[Advection–diffusion on an oscillating ellipsoid]
\label{ex_ev1}

To facilitate comparison with existing methods and further validate the proposed approach, we consider the advection–diffusion equation \eqref{eq:ev-u} on an oscillating ellipsoid (\cite{Petras2016evolving, Hsu2019evolving, Hu2024pinn}), represented as the zero-level set
\[
\Gamma(t) = \left\{(x, y, z) \mid \left( \frac{x}{1.5a(t)} \right)^2 + y^2 + \left( \frac{z}{0.5} \right)^2 = 1 \right\},
\]
where $a(t) = \sqrt{1 + 0.95 \sin(\pi t)}$. The associated velocity field is $\bv = \left(\frac{a'(t)}{a(t)} x, 0, 0 \right)$. We adopt the exact solution $u^*(x,y,z) = \sin(x + \sin(t))\exp(\cos(y-z))$, from which the source term and initial condition follow directly.
\end{example}

First, we approximate the surface evolution by learning the flow map $\mathcal N_e$.
In this benchmark the velocity field $\bv=(a'(t)/a(t)\,x,0,0)$ induces a pure stretching in the $x$-direction,
so the trajectory equations decouple by components. Consequently, each component network $\mathcal N_e^i$ can be taken to depend
only on $(t,(X_0)_i)$, i.e., it has two inputs.
We use three one-hidden-layer RaNNs with widths $(M_1,M_2,M_3)$ to approximate the three coordinate components of the flow map.
We set $r_x^{1}=r_x^{2}=r_x^{3}=1$, and choose $r_t^{1}=16$, $r_t^{2}=r_t^{3}=1$ to reflect the different temporal scales.
Based on the coordinate ranges of the initial surface $\Gamma_0$ we use the axis-aligned bounding boxes
$\widetilde D_1=[-1.5,1.5]$, $\widetilde D_2=[-1,1]$, and $\widetilde D_3=[-0.5,0.5]$, and take
$\beta=100$ for enforcing the initial condition.

To sample points on the initial surface approximately uniformly, we employ the Fibonacci lattice method \cite{Gonzalez2010Fibonacci}.
We select $N_x$ points on $\Gamma_0$ to enforce the initial condition and $N_{\rm col}=N_tN_x$ space--time points to enforce the
trajectory ODE residual, where $N_t$ is the number of time samples in $I=[0,2]$.

We evaluate accuracy at $t=2$ using relative discrete $\ell^2$ errors in the flow-map coordinates, unit normals, and mean curvature:
\[
E_{\bx}:=\frac{\|\mathcal N_e-\bx\|_{\ell^2(P_{\rm test})}}{\|\bx\|_{\ell^2(P_{\rm test})}},\qquad
E_{\bn}:=\frac{\|\bn_{\mathcal N}-\bn\|_{\ell^2(P_{\rm test})}}{\|\bn\|_{\ell^2(P_{\rm test})}},\qquad
E_{H}:=\frac{\|H_{\mathcal N}-H\|_{\ell^2(P_{\rm test})}}{\|H\|_{\ell^2(P_{\rm test})}},
\]
together with the total training time, where $P_{\rm test}\subset\Gamma(2)$ is an independent test set.
The results are summarized in Table~\ref{tablel_ev1_x}, and the evolving surface is visualized in Figure~\ref{fig_ev1_x}, with color indicating the mean curvature.

\begin{table}[H]
\centering
\begin{tabular}{cccccc}
\toprule[1.5pt]
$M_1, M_2, M_3$              & $N_{\rm col},N_x$ & $E_{\bx}$ & $E_{\bn}$    & $E_H$    & Time (s)  \\ \midrule[1.5pt]
\multirow{4}{*}{1600,500,500} & $80^2,80$               & 2.27E-06  & 1.58E-06 & 2.42E-06 & 0.624  \\
                      & $120^2,120$             & 7.48E-08  & 1.17E-07 & 6.62E-07 & 1.324  \\
                      & $160^2,160$             & 8.06E-09  & 8.63E-08 & 6.47E-07 & 2.150  \\
                      & $200^2,200$             & 5.31E-09  & 3.64E-08 & 1.81E-07 & 3.510  \\ \midrule[1.5pt]
\multirow{4}{*}{2600,500,500} & $80^2,80$               & 5.87E-05  & 1.11E-05 & 1.71E-05 & 1.125  \\
                      & $120^2,120$             & 1.32E-07  & 1.27E-07 & 1.88E-07 & 2.136  \\
                      & $160^2,160$             & 1.46E-09  & 2.41E-09 & 1.08E-08 & 3.554  \\
                      & $200^2,200$             & 2.42E-10  & 2.54E-09 & 1.36E-08 & 5.293 
 \\ \bottomrule[1.5pt]
\end{tabular}
\caption{Relative errors of $\bx$, $\bn$, and $H$ at $t=2$ and training time for the learned surface evolution.}
\label{tablel_ev1_x}
\end{table}

\begin{figure}[H] 
	\centering  
	\begin{subfigure}{0.23\textwidth}
		\centering      
		\includegraphics[width=1.0\textwidth]{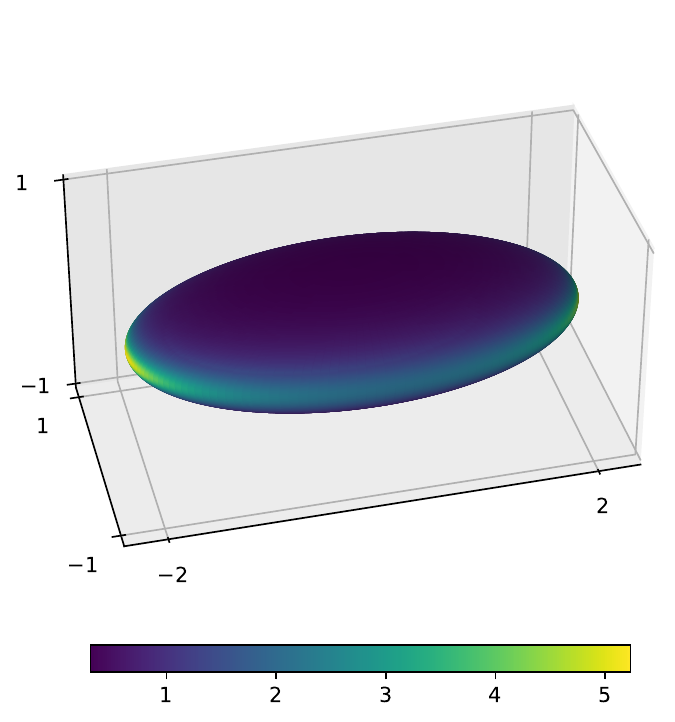}  
		\caption{t = 0.5.}
		%\label{fig:a}
	\end{subfigure}
	\begin{subfigure}{0.23\textwidth}
		\centering      
		\includegraphics[width=1.0\textwidth]{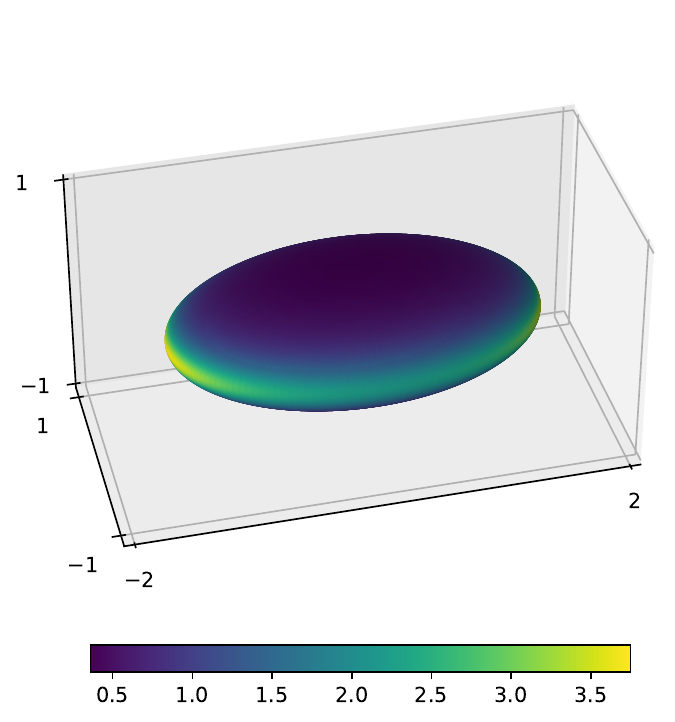}  
		\caption{t = 1.0.}
		%\label{fig:b}
	\end{subfigure}
	\begin{subfigure}{0.23\textwidth}
			\centering      
			\includegraphics[width=1.0\textwidth]{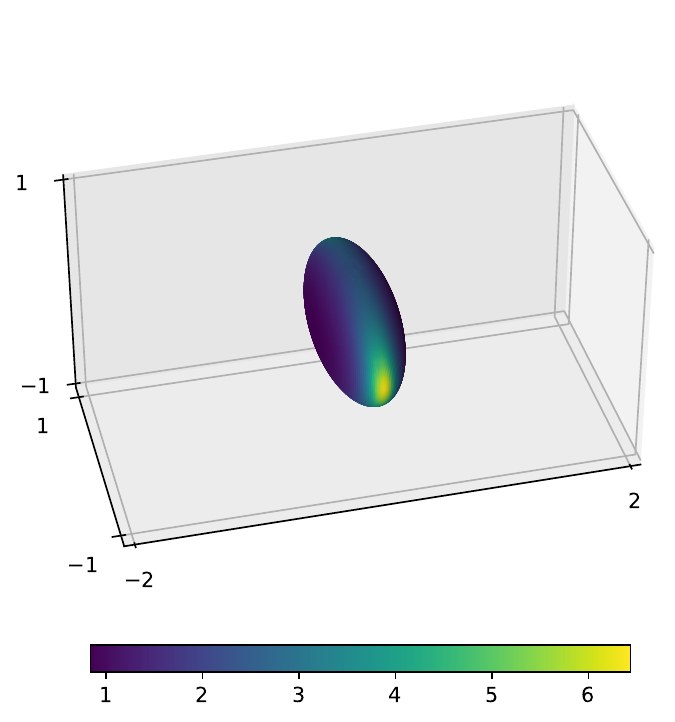}  
			\caption{t = 1.5.}
			%\label{fig:b}
		\end{subfigure}
	\begin{subfigure}{0.23\textwidth}
			\centering      
			\includegraphics[width=1.0\textwidth]{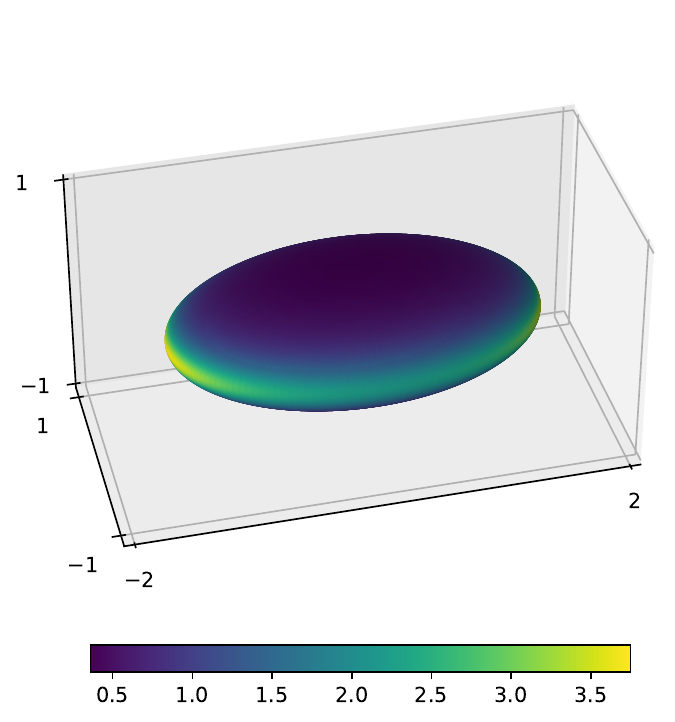}  
			\caption{t = 2.0.}
			%\label{fig:b}
		\end{subfigure}
	\caption{Evolving ellipsoid and corresponding mean curvature distribution.}  
	\label{fig_ev1_x}
\end{figure}

Next, we solve the advection--diffusion problem \eqref{eq:ev-u} on the evolving surface $\Gamma(t)$.
Using the learned evolution networks $(\mathcal N_e^{1},\mathcal N_e^{2},\mathcal N_e^{3})$ with widths $(2600,500,500)$ and the
sampling resolution $N_0=200$ (as reported in Table~\ref{tablel_ev1_x}), we generate space--time collocation points on
$I\times\Gamma(t)$ to train the solution network $u_e$.
We set $r_t=r_x=0.8$ and use an initial-condition penalty $\beta=100$.
The bias sampling set is chosen by drawing $\bB\sim U([0,2]\times[-2,2]\times[-1,1]\times[-0.5,0.5])$.
The relative $\ell^2$ errors (evaluated at $t=2$) and training times are summarized in Table~\ref{tablel_ev1_u},
demonstrating that the proposed method can efficiently and accurately solve advection--diffusion equations on evolving surfaces.

\begin{table}[H]
\centering
\begin{tabular}{ccccccc}
\toprule[1.5pt]
$M$                & \multicolumn{2}{c}{1400} & \multicolumn{2}{c}{2000} & \multicolumn{2}{c}{2600} \\ \cmidrule[0.5pt](l){2-3} \cmidrule[0.5pt](l){4-5} \cmidrule[0.5pt](l){6-7}
$N_x,N_{\rm col}$ & Error       & Time (s)    & Error       & Time (s)    & Error       & Time (s)    \\ \midrule[0.5pt]
$80,80^2$               & 2.41E-04    & 1.140   & 9.63E-05    & 1.714   & 7.80E-05    & 2.370   \\
$120,120^2$             & 2.83E-05    & 2.218   & 8.16E-06    & 3.301   & 1.03E-05    & 4.484   \\
$160,160^2$             & 4.66E-06    & 3.860   & 3.05E-06    & 5.613   & 1.40E-06    & 7.532   \\
$200,200^2$             & 3.32E-06    & 5.901   & 6.97E-07    & 8.579   & 3.07E-07    & 11.330   \\
\bottomrule[1.5pt]
\end{tabular}
\caption{Relative errors of $u$ at $t=2$ on the evolving surface and training time.}
\label{tablel_ev1_u}
\end{table}

\begin{example}[Surfactant transport on a droplet surface under shear flow]
\label{ex_ev2}
We consider surfactant transport on a droplet surface driven by a prescribed shear flow, following \cite{Hu2024pinn}.
Fluid effects are neglected and the ambient velocity field is
\[
\bv(t,x,y,z)=(z,0,0).
\]
The initial droplet surface is the unit sphere centered at the origin, and the evolving surface admits the exact level-set form
\[
\Gamma(t)=\big\{(x,y,z)\in\mathbb R^3:\ (x-tz)^2+y^2+z^2=1\big\},\qquad t\in I=[0,3].
\]
We consider the conservative surface advection--diffusion equation \eqref{eq:ev-u} with $f=0$ and initial condition
$u(0,\cdot)\equiv 1$ on $\Gamma(0)$.
\end{example}

Although the exact flow map is available for this benchmark, we first approximate the surface evolution using the RaNN flow-map
networks $\mathcal N_e=(\mathcal N_e^1,\mathcal N_e^2,\mathcal N_e^3)$ to assess geometric accuracy.
For $\bv=(z,0,0)$ the trajectory equations decouple as
$x(t)=x_0+t z_0$, $y(t)=y_0$, $z(t)=z_0$, hence we take
\[
\mathcal N_e^1=\mathcal N_e^1(t,x_0,z_0),\qquad
\mathcal N_e^2=\mathcal N_e^2(t,y_0),\qquad
\mathcal N_e^3=\mathcal N_e^3(t,z_0),
\]
i.e., the spatial input dimensions are $d_1=2$, $d_2=d_3=1$.
We set $r_x^i=r_t^i=1$ and enforce the initial condition using the penalty parameter $\beta=100$.
Since $\Gamma(0)$ is the unit sphere, we sample the bias sets as
\[
\bB_1\sim U([0,3]\times[-1,1]^2),\qquad
\bB_2\sim U([0,3]\times[-1,1]),\qquad
\bB_3\sim U([0,3]\times[-1,1]).
\]
With $N_0$ samples on $\Gamma(0)$ and $N_0$ samples in time, we form $N_{\rm col} = N_0^2$ space--time collocation pairs to enforce the
trajectory ODE residual, and report relative errors of $\bx$, $\bn$, and $H$ at $t=3$
together with the training time in Table~\ref{tablel_ev2_x}.

\begin{table}[H]
\centering
\begin{tabular}{cccccc}
\toprule[1.5pt]
$M_1, M_2, M_3$              & $N_{\rm col},N_0$ & $E_{\bx}$ & $E_{\bn}$    & $E_H$    & Time (s)  \\ \midrule[1.5pt]
\multirow{4}{*}{1600} & $80^2,80$               & 1.11E-06  & 7.49E-06 & 3.36E-05 & 0.975 \\
                      & $120^2,120$             & 8.43E-10  & 7.98E-09 & 5.50E-08 & 2.046 \\
                      & $160^2,160$             & 1.19E-09  & 1.19E-08 & 6.50E-08 & 3.754 \\
                      & $200^2,200$             & 8.44E-10  & 1.19E-08 & 8.77E-08 & 6.137 \\ \midrule[1.5pt]
\multirow{4}{*}{2600} & $80^2,80$               & 1.90E-06  & 1.43E-05 & 6.02E-05 & 1.805 \\
                      & $120^2,120$             & 2.18E-08  & 2.22E-07 & 1.57E-06 & 3.685 \\
                      & $160^2,160$             & 2.59E-10  & 3.17E-09 & 2.05E-08 & 6.037 \\
                      & $200^2,200$             & 4.54E-11  & 5.32E-10 & 3.71E-09 & 10.387 \\ \bottomrule[1.5pt]
\end{tabular}
\caption{Relative errors of $\bx$, $\bn$, and $H$ at $t=3$ and training time.}
\label{tablel_ev2_x}
\end{table}

Because the prescribed shear flow $\bv$ is incompressible, the enclosed droplet volume should be conserved in time.
We compute the volume using the standard surface formula
\[
V(t)=\frac{1}{3}\int_{\Gamma(t)} \bx\cdot \bn \,ds,
\]
so that $V(0)=\frac{4}{3}\pi$ for the unit sphere.
We evaluate the relative volume error
\[
E_V(t):=\frac{|V_{\mathcal N}(t)-V(0)|}{V(0)}
\]
by Gaussian quadrature on the reconstructed surface, where $\Gamma_{\mathcal N}(t)$ denotes the surface represented by $\mathcal N_e$.
Figure~\ref{fig_ev2_V} compares the resulting volume error with that reported in \cite{Hu2024pinn}, indicating that the proposed method
achieves substantially smaller volume drift for this benchmark.

\begin{figure}[htbp] 
	\centering  
	\begin{subfigure}{0.45\textwidth}
		\centering      
		\includegraphics[width=0.9\textwidth]{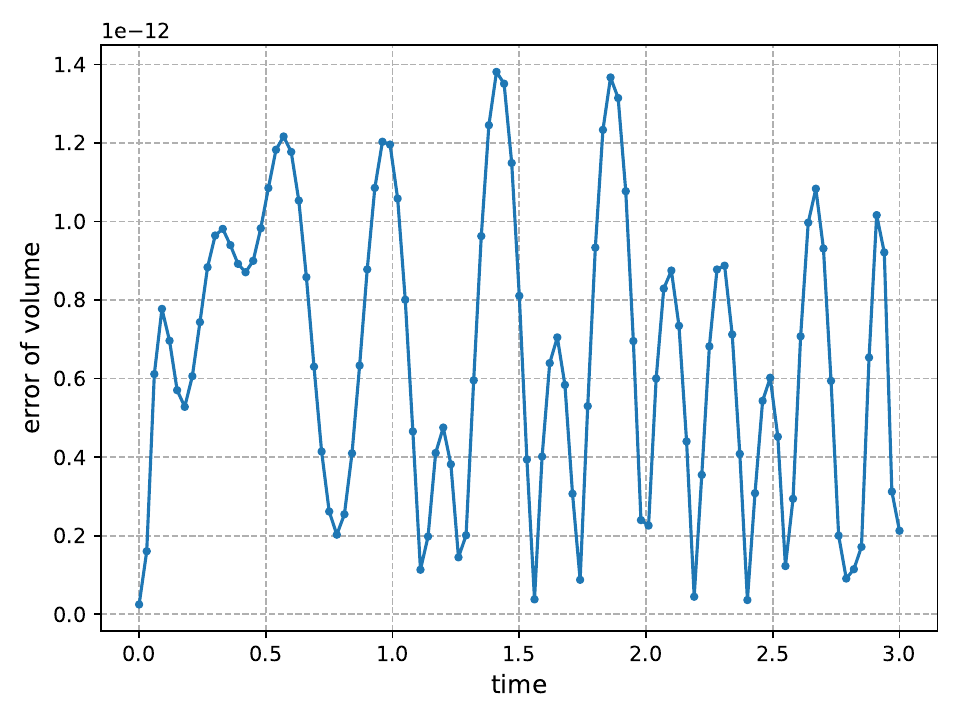}  
		\caption{Droplet volume error.}
		\label{fig_ev2_V}
	\end{subfigure}\quad
	\begin{subfigure}{0.45\textwidth}
		\centering      
		\includegraphics[width=0.9\textwidth]{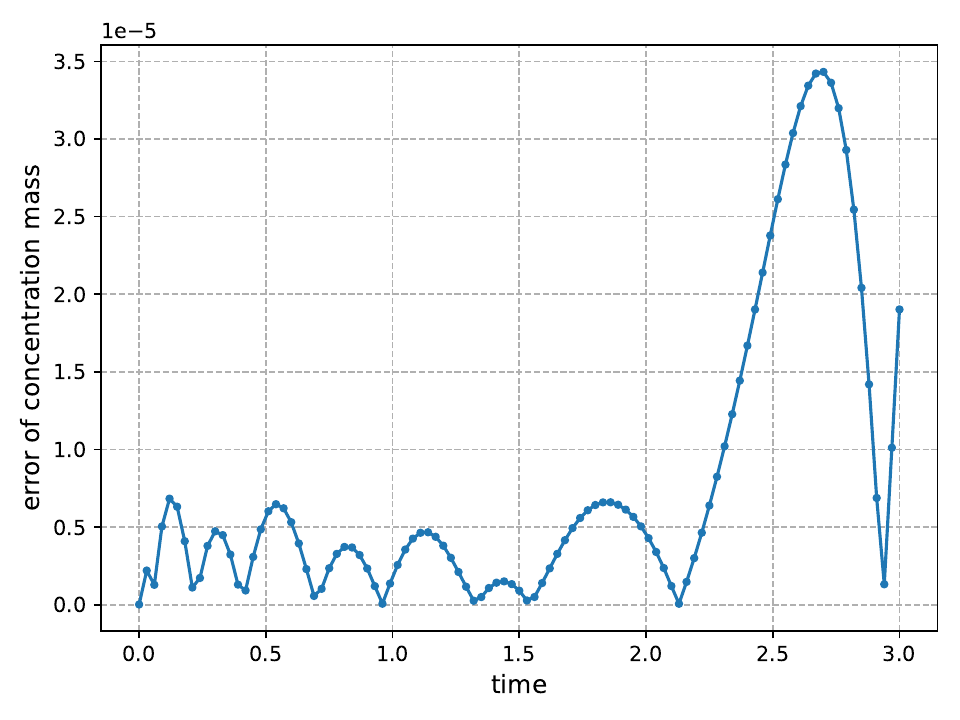}  
		\caption{Surfactant mass error.}
		\label{fig_ev2_M}
	\end{subfigure}
	\caption{Time histories of droplet volume error (a) and surfactant mass error (b).}  
	%\label{cup_heat}
\end{figure}

\begin{figure}[htbp]
    \centering
    \begin{subfigure}{0.3\textwidth}
        \includegraphics[width=\linewidth]{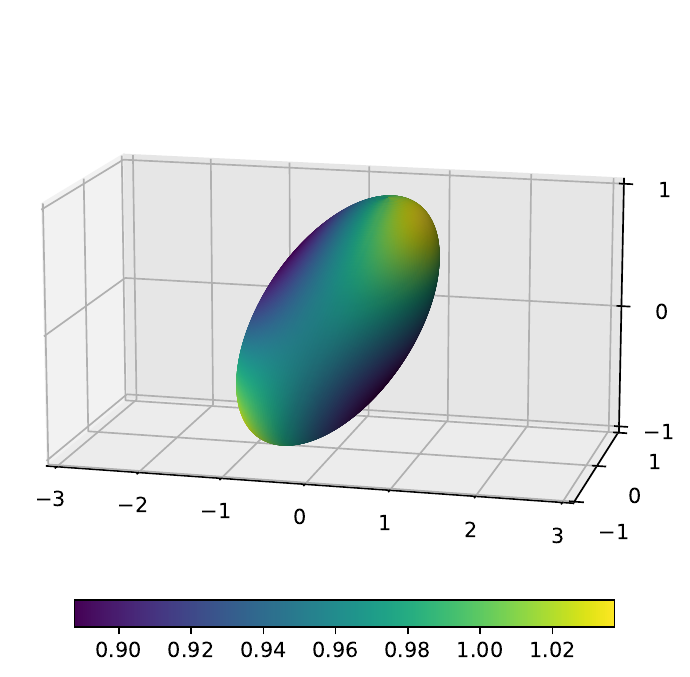}
        \caption{t=0.6.}
    \end{subfigure}\hfill
    \begin{subfigure}{0.3\textwidth}
        \includegraphics[width=\linewidth]{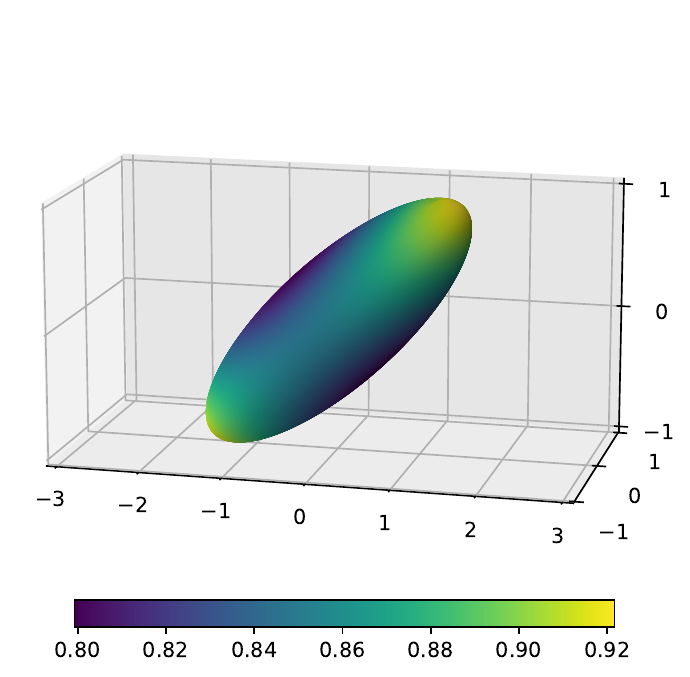}
        \caption{t=1.2.}
    \end{subfigure}\hfill
    \begin{subfigure}{0.3\textwidth}
        \includegraphics[width=\linewidth]{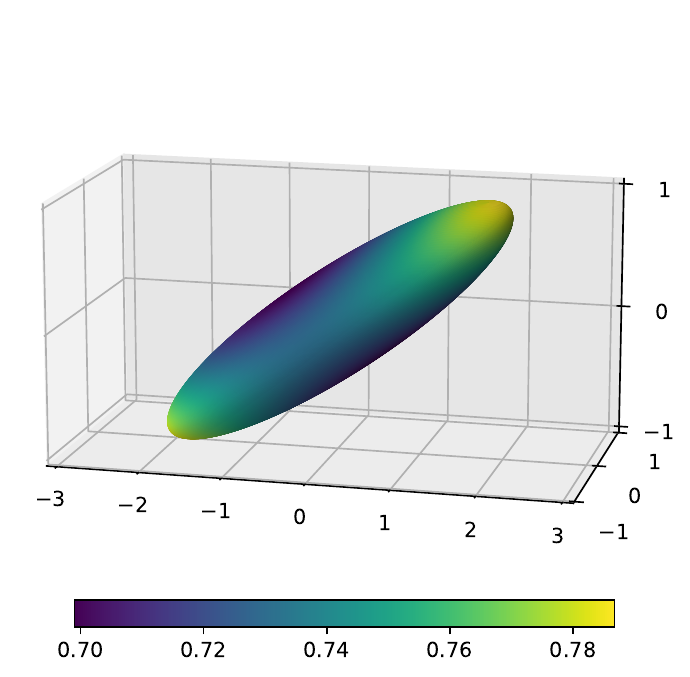}
        \caption{t=1.8.}
    \end{subfigure} \\ 
    \begin{subfigure}{0.3\textwidth}
        \includegraphics[width=\linewidth]{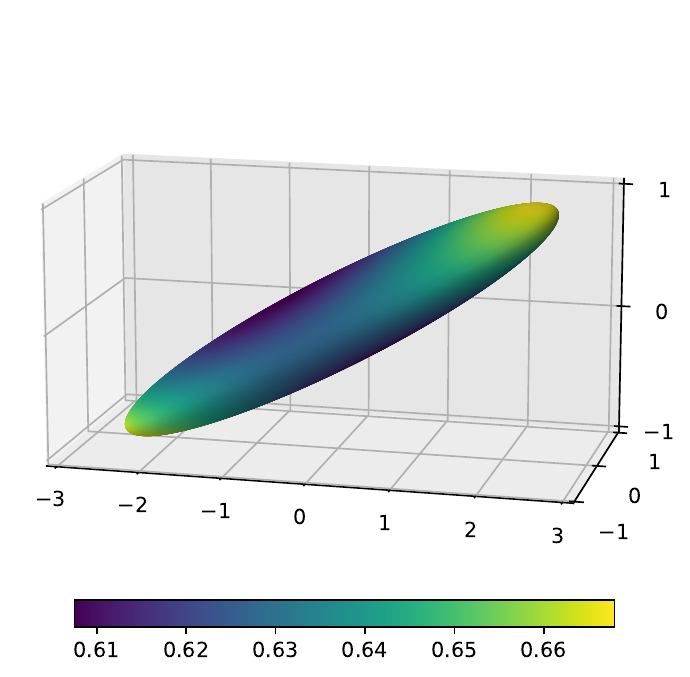}
        \caption{t=2.4.}
    \end{subfigure}
    \begin{subfigure}{0.3\textwidth}
        \includegraphics[width=\linewidth]{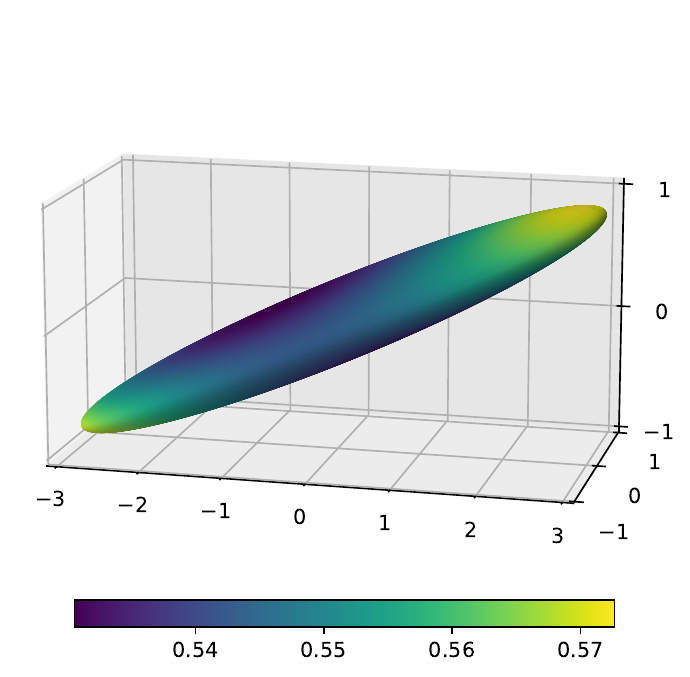}
        \caption{t=3.0.}
    \end{subfigure}
    \caption{Snapshots of the droplet surface $\mathcal{N}^{1,2,3}_e$ and surfactant concentration $u_e$ at different times. Colors indicate the magnitude of $u_e$.}
\label{fig_ev2_u}
\end{figure}

We next solve \eqref{eq:ev-u} on $\Gamma(t)$ using the learned evolution network with width $M_i=2600, i=1,2,3$, and sampling resolution
$N_0=200$ (cf.\ Table~\ref{tablel_ev2_x}) to generate space--time collocation points on $I\times\Gamma(t)$.
We set $r_t=r_x=1$ and use an initial-condition penalty $\beta=100$ for the solution network $u_e$.
The bias sampling set is chosen as $\bB\sim U([0,3]\times[-3.2,3.2]\times[-1,1]\times[-1,1])$.
Since $f=0$, the total surfactant mass
\[
m(t):=\int_{\Gamma(t)} u\,ds
\]
should be conserved with $m(0)=4\pi$. We therefore monitor the relative mass error
$E_m(t):=|m_{\mathcal N}(t)-m(0)|/m(0)$ (Gaussian quadrature), reported in Figure~\ref{fig_ev2_M}.
Figure~\ref{fig_ev2_u} shows snapshots of the predicted concentration; the shear flow advects surfactant toward the droplet tips.
Notably, even without explicit enforcement of mass conservation, the proposed method preserves surfactant mass to high accuracy.

Traditional mesh-based methods for PDEs on evolving surfaces typically require repeated mesh updates to accommodate large deformations,
together with mesh-to-mesh data transfer (interpolation) of geometric quantities and solution values, which increases computational
overhead and may introduce additional numerical errors.
By contrast, the proposed RaNN framework is mesh-free in the sense that it enforces the governing equations directly at space--time
collocation points and represents the evolving geometry through a learned flow map, avoiding remeshing and solution transfer between
successive meshes.
In the present shear-flow benchmark, the method remains accurate under significant surface distortion and preserves key invariants
(e.g., volume and surfactant mass) to high accuracy, underscoring its potential for solving PDEs on evolving surfaces that undergo
smooth, topology-preserving transformations.

\section{Summary}
\label{sec:conclusion}

In this work, we developed a mesh-free RaNN framework for solving PDEs on static and evolving surfaces using the strong form of the PDE.
The proposed methodology accommodates a broad range of surface representations, including (i) parametrizable surfaces (via an atlas),
(ii) implicit surfaces described by level-set functions, and (iii) surfaces provided only as unstructured point clouds.

For PDEs on static surfaces, we constructed RaNN discretizations for both stationary and time-dependent problems and provided a
corresponding theoretical analysis for the parametrization-based formulation with interface compatibility.
Across a variety of benchmarks, the proposed approach achieved high accuracy with low computational cost, and compared favorably with
recent neural-network-based solvers for surface PDEs, including the PINN-type method in \cite{Hu2024pinn}.

For evolving surfaces with topology preserved over time, we further introduced a RaNN-based strategy that learns the surface evolution
through a flow-map representation and then solves the surface PDE on the resulting space--time collocation set.
This avoids repeated mesh construction/remeshing and mesh-to-mesh solution transfer, thereby reducing computational overhead and
mitigating errors associated with low-quality surface discretizations.
Numerical experiments demonstrate that the proposed method can accurately capture surface evolution and efficiently approximate PDE
solutions on strongly deformed evolving surfaces.

Future work will extend the framework in several directions.
First, we will address evolving surfaces with topological changes, where a single global flow-map
representation may break down and alternative implicit or hybrid geometric representations become necessary.
Second, we will study strongly coupled evolution problems in which the surface motion and the surface PDE are mutually dependent,
so that the geometry and the solution must be advanced in a coupled manner rather than in two separate stages.
Third, we will consider more complex models such as phase-field and multiphysics systems posed on surfaces.
In addition, we plan to improve robustness and adaptability by developing principled strategies for selecting bandwidth parameters
(e.g., $r_x$ and $r_t$) and for designing effective space--time sampling and conditioning control in the resulting least-squares systems.

\end{document}